\documentclass[aap]{imsart}

\RequirePackage{amsthm,amsmath,amsfonts,amssymb,mathtools}
\RequirePackage[authoryear]{natbib}
\RequirePackage[colorlinks,pagebackref,citecolor=blue,urlcolor=blue]{hyperref}
\RequirePackage{graphicx}
\usepackage{enumerate}

\startlocaldefs
\numberwithin{equation}{section}
\theoremstyle{plain}
\newtheorem{theorem}{Theorem}[section]
\newtheorem{assumption}[theorem]{Assumption}
\newtheorem{lemma}[theorem]{Lemma}

\newtheorem{proposition}[theorem]{Proposition}
\newtheorem{corollary}[theorem]{Corollary}
\newtheorem{remark}[theorem]{Remark}
\newcommand{\brac}[1]{\langle#1\rangle}\newcommand{\abs}[1]{\lvert#1\rvert}\newcommand{\Abs}[1]{\left|#1\right|}\newcommand{\ABs}[1]{\biggl|#1\biggr|}\newcommand{\Brace}[1]{\left\{#1\right\}}\newcommand{\paren}[1]{\left(#1\right)}\newcommand{\Paren}[1]{\biggr(#1\biggl)}\newcommand{\Square}[1]{\left[#1\right]}\newcommand{\SQuare}[1]{\biggl[#1\biggr]}\newcommand{\ocinterval}[1]{(#1]}\newcommand{\cointerval}[1]{[#1)}
\newcommand{\ov}[1]{\overline{#1}}\newcommand{\wt}[1]{\widetilde{#1}}\newcommand{\wh}[1]{\widehat{#1}}\newcommand{\deq}{\overset{\mathrm{d}}{=}}
\usepackage{etoolbox}
\renewcommand{\do}[1]{\csdef{b#1}{\mathbb{#1}}}\docsvlist{A,B,C,D,E,F,G,H,I,J,K,L,M,N,O,P,Q,R,S,T,U,V,W,X,Y,Z}
\renewcommand{\do}[1]{\csdef{c#1}{\mathcal{#1}}}\docsvlist{A,B,C,D,E,F,G,H,I,J,K,L,M,N,O,P,Q,R,S,T,U,V,W,X,Y,Z}
\renewcommand{\do}[1]{\csdef{f#1}{\mathfrak{#1}}}\docsvlist{A,B,C,D,E,F,G,H,I,J,K,L,M,N,O,P,Q,R,S,T,U,V,W,X,Y,Z}

\renewcommand{\d}{\mathrm{d}}\newcommand{\ds}{\mathrm{d}s}\newcommand{\dt}{\mathrm{d}t}\newcommand{\du}{\mathrm{d}u}\newcommand{\dx}{\mathrm{d}x}
\newcommand{\R}{\mathbb{R}}

\newcommand{\tpar}{{\!\scalebox{0.6}{\rotatebox{30}{$\parallel$}}}}
\newcommand{\RelaxCommands}[1]{\renewcommand*{\do}[1]{\expandafter\let\csname ##1\endcsname\relax}\docsvlist{#1}}
\RelaxCommands{P}
\newcommand{\DeclareMathOperators}[1]{\renewcommand*{\do}[1]{\expandafter\DeclareMathOperator\csname ##1\endcsname{##1}}\docsvlist{#1}}
\DeclareMathOperators{P,E,V,Var,C,Cov,id,erf,erfc,erfcx,Unif,ESS,MSE,TV}
\endlocaldefs

\begin{document}

\begin{frontmatter}
\title{Diffusive Scaling Limits of Forward Event-Chain Monte Carlo: Provably Efficient Exploration with Partial Refreshment}
\runtitle{Diffusive Scaling Limits of FECMC}

\begin{aug}
\author[A]{\fnms{Hirofumi}~\snm{Shiba}\ead[label=e1]{shiba.hirofumi@ism.ac.jp}\orcid{0009-0007-8251-1224}}
\and
\author[A]{\fnms{Kengo}~\snm{Kamatani}\ead[label=e2]{kamatani@ism.ac.jp}\orcid{0000-0002-1537-3446}}
\address[A]{Institute of Statistical Mathematics\printead[presep={,\ }]{e1,e2}}

\end{aug}

\begin{abstract}
    Piecewise deterministic Markov process samplers are attractive alternatives to Metropolis--Hastings algorithms. A central design question is how to incorporate partial velocity refreshment to ensure ergodicity without injecting excessive noise. Forward Event-Chain Monte Carlo (FECMC) is a generalization of the Bouncy Particle Sampler (BPS) that addresses this issue through a stochastic reflection mechanism, thereby reducing reliance on global refreshment moves. Despite promising empirical performance, its theoretical efficiency remains largely unexplored.

    We develop a high-dimensional scaling analysis for standard Gaussian targets and prove that the negative log-density (or potential) process of FECMC converges to an Ornstein--Uhlenbeck diffusion, under the same scaling as BPS. We derive closed-form expressions for the limiting diffusion coefficients of both methods by analyzing their associated radial momentum processes and solving the corresponding Poisson equations. These expressions yield a sharp efficiency comparison: the diffusion coefficient of FECMC is strictly larger than that of optimally tuned BPS, and the optimum for FECMC is attained at zero global refreshment. Specifically, they imply an approximately eightfold increase in effective sample size per event over optimal BPS. Numerical experiments confirm the predicted diffusion coefficients and show that the resulting efficiency gains remain substantial for a range of non-Gaussian targets. Finally, as an application of these results, we propose an asymptotic variance estimator for Piecewise deterministic Markov processes that becomes increasingly efficient in high dimensions by extracting information from the velocity variable.
\end{abstract}



\end{frontmatter}


\section{Introduction}

Since the introduction of the random walk algorithm by \cite{Metropolis+1953}, Markov chain Monte Carlo (MCMC) methods have found widespread applications in physics and statistics, particularly in Bayesian statistics following the advocacy of \cite{Gelfand-Smith1990}. While the Metropolis--Hastings (MH) paradigm introduced by \cite{Hastings1970} has underpinned the majority of applications, a fundamentally different class of algorithms based on piecewise deterministic Markov processes (PDMPs) \citep{Davis1984} has emerged as a viable alternative. These methods originated in computational physics and molecular simulation as rejection-free, event-driven algorithms designed to overcome the diffusive behavior inherent in MH-type dynamics \citep{Bernard+2009,Peters-deWith2012,Michel+2014,Faulkner+2018,Krauth2021} and were later adopted in the statistical literature \citep{Vanetti+2018,Bouchard-Cote+2018,Bierkens+2019,Fearnhead+2018,Faulkner-Livingstone2024}. Algorithms designed within the PDMP framework enjoy two key structural advantages that are typically lost in MH-based methods. First, PDMPs are inherently irreversible, and this lack of detailed balance has been shown to accelerate convergence relative to reversible counterparts, a well-documented fact in both computational physics \citep{Nishikawa+2015,Lei+2019} and statistics \citep{Andrieu-Livingstone2021,Eberle-Lorler2024}.
Second, and crucially for modern large-scale inference, the PDMP framework naturally accommodates unbiased subsampling and stochastic gradients while preserving the exactness of the target distribution \citep{Bierkens+2019,Sen+2020,Fearnhead+2024}. This stands in contrast to most stochastic gradient MH methods, which typically introduce asymptotic bias in exchange for scalability \citep{Welling-Teh2011,Chen+2014,Nemeth-Fearnhead2021}. 
Other structural benefits of PDMP-based samplers, including robustness in multiscale targets and bespoke implementation for special targets, are also under active investigation \citep{Bierkens+2023,Chevallier+2023}.

Despite these advantages, a precise understanding of the high-dimensional scaling behavior of PDMP-based samplers remains incomplete. In particular, although several PDMP algorithms have been proposed, there is still no fully unified theoretical framework for comparing their efficiency or guiding practical algorithmic choices in high dimensions, despite the fact that scaling analysis is well established for MH-type methods; see Section~\ref{subsec-ScalingAnalysis}. In this work, we address this gap by deriving high-dimensional diffusive scaling limits, which enable a theoretical comparison of their computational complexity and asymptotic efficiency. We focus on two representative and widely used methods: the Bouncy Particle Sampler (BPS) \citep{Bouchard-Cote+2018} and the Forward Event-Chain Monte Carlo (FECMC) \citep{Michel+2020}. Both methods employ piecewise linear dynamics perturbed by directional changes at random event times. A key difference is how randomness is injected via their directional changes: BPS employs deterministic reflections and thus relies on global refreshment moves for randomization, whereas FECMC achieves it through a stochastic reflection mechanism. In BPS, there is a tension in selecting the Poisson rate $\rho$ for the global refreshment events. While global refreshment is essential for ergodicity, it may also inject additional noise that can hinder efficient exploration of the target distribution. Therefore, the primary design goal of FECMC was to discard the global refreshment while preserving ergodicity. This is achieved by stochastic reflection that combines (1) partial refreshment of the velocity component parallel to the gradient and (2) minimal randomization of the orthogonal component. The effectiveness of this strategy has been demonstrated through numerical experiments in \cite{Michel+2020}. However, a theoretical explanation for this observed improvement is lacking. In particular, it remains unclear whether additional global refreshment ($\rho>0$) would accelerate exploration.

We provide a theoretical explanation by directly comparing the high-dimensional diffusion scaling limits between FECMC and BPS. In Section~\ref{sec-ScalingAnalysis}, we establish the weak convergence theorems for the FECMC radial momentum process and the negative log-density (or potential) process. We start with the case $\rho=0$, which requires new techniques. In this case, the FECMC process loses the regeneration structure on which the martingale arguments of \cite{Bierkens+2022} crucially depend. By considering the dual predictable projection instead of the generator, we derive the limit theorem for the case of $\rho=0$ through a semimartingale technique \citep{Metivier1982,Jacod-Shiryaev2003}. We note that the case $\rho=0$ has not been considered since BPS fails to be ergodic under this choice. Once the case $\rho=0$ is understood, the extension to $\rho>0$ follows naturally by combining our approach with the techniques developed in \cite{Bierkens+2022} to analyze BPS. For the standard Gaussian targets, this limit turns out to be an Ornstein--Uhlenbeck (OU) process
\[\d Y_t=-\frac{\sigma^2(\rho)}{4}Y_t\,\dt+\sigma(\rho)\,\d B_t,\]
where only the diffusivity parameter $\sigma^2$ differs between FECMC and BPS. 
Therefore, a comparison of $\sigma$ will enable us to compare the limiting speeds of the two processes and ultimately the asymptotic efficiency of the two algorithms. However, the previous expression for $\sigma$ of BPS obtained in \cite{Bierkens+2022} involves an intractable integral, which hinders direct comparison. In \cite{Bierkens+2022}, the authors approximated this quantity numerically and proposed using it to tune the refreshment rate $\rho$ by maximizing diffusivity $\sigma^2$. In Section~\ref{sec-Diffusivity}, we derive explicit analytic formulae for the diffusion coefficients of both BPS and FECMC (Theorem~\ref{thm-formula-for-sigma}). We plot the result in Figure~\ref{fig-sigma}, which illustrates the dominance of FECMC over BPS. This is achieved by reinterpreting the intractable integral as an expectation involving the associated resolvent operator of the radial momentum process and explicitly solving the associated Poisson equations. Interestingly, our results suggest that increasing $\sigma^2$, and hence speeding up the limiting negative log-density process $Y$, can be achieved by improving the spectral properties of the radial momentum process. Consequently, we analytically establish that FECMC achieves a higher speed than even a perfectly tuned BPS in sufficiently high dimensions. Additionally, we conclude that FECMC is most efficient when the refreshment rate $\rho$ is set to zero, and any additional noise will only deteriorate performance. This confirms that the stochastic reflection mechanism of FECMC successfully preserves the necessary amount of randomization, thereby enabling efficient exploration while remaining free of hyperparameter $\rho$ and computationally cheaper. We support our analysis and its robustness through numerical experiments in Section~\ref{sec-Experiment}.

\begin{figure}[htbp]\centering
    \includegraphics[width=0.45\textwidth]{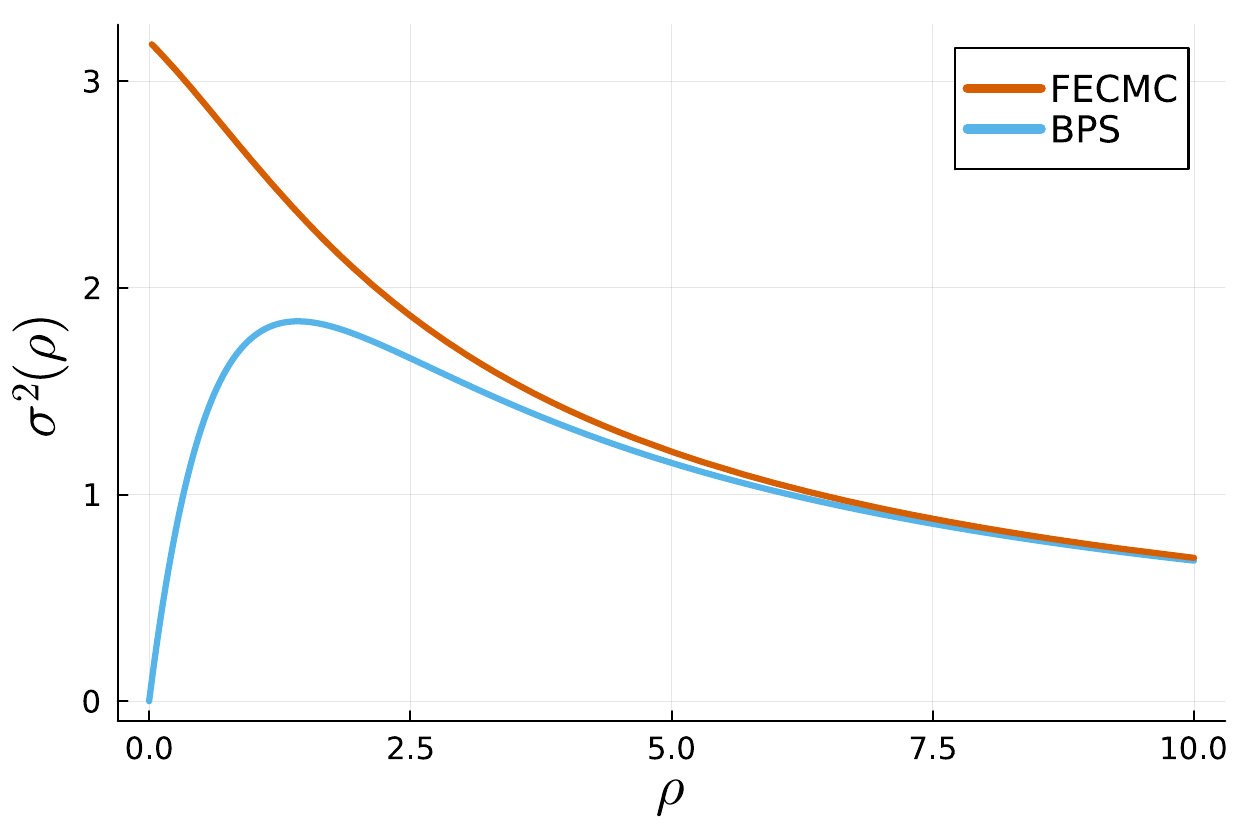}
    \caption{Diffusion coefficient $\sigma$ of the limiting negative log-density (potential) process vs. the global refreshment rate $\rho$ (FECMC vs. BPS). Standard Gaussian target with spherical velocity. Curves computed from the analytic expressions in Theorem~\ref{thm-formula-for-sigma}.
    }\label{fig-sigma}
\end{figure}

Finally, our analysis implies novel insights into PDMP output analysis. In practice, the velocity (or momentum) variable $v$ is typically treated as a purely auxiliary component and is thus discarded after simulation. However, we argue that $v$ may serve as a fast proxy for convergence diagnosis when the mixing of $x$ is too slow. Specifically, in our setting, the mixing time of the negative log-density $U(X_t)$ scales as $O(d)$, whereas its time derivative $(\nabla U(X_t)|V_t)$ mixes on an $O(1)$ timescale. This separation of timescales suggests that asymptotic variance estimation based on the time derivative can be more efficient in high dimensions than applying the same estimator directly to $U$. We formalize this observation as a corollary to our main results in Section~\ref{subsec-fast-proxy} and illustrate it empirically in Section~\ref{subsec-experiment-BM}, focusing on the commonly used batch means estimator.

\section{Background and Setting}
\subsection{Diffusion Scaling Limits for MCMC Algorithms}\label{subsec-ScalingAnalysis}

In Monte Carlo applications, the dimension $d$ of the state space can be very large, especially in modern data science and machine learning. A central question is thus how the efficiency of a given algorithm scales with $d$. In the statistics literature, the scaling analysis approach based on diffusion approximations was introduced in \cite{Roberts+1997} to analyze the high-dimensional behavior of MCMC samplers. The authors derived a weak convergence limit for finite dimensional projections of the random walk MH algorithm under appropriate time scaling. The limiting diffusion process captures effective first-order dynamics on a macroscopic scale. Since then, scaling analysis has been applied to algorithms beyond the random-walk MH algorithm \citep{Roberts+1997,Mattingly+2012}, such as the Metropolis-adjusted Langevin algorithm \citep{Roberts-Rosenthal1998,Pillai+2012} and Hamiltonian Monte Carlo \citep{Beskos+2013}.

While the diffusion limit theorems for MH-type algorithms are highly advanced \citep{Roberts-Rosenthal2001,Sherlock+2015,Roberts-Rosenthal2016,Zanella+2017,Bierkens-Roberts2017,Kamatani2018,Kuntz+2019,Kamatani2020,Yang+2020}, those for PDMP algorithms are still limited. For the high-dimensional asymptotic regime $d\to\infty$, there are only two works, to our knowledge, \cite{Deligiannidis+2021} and \cite{Bierkens+2022}, that directly consider PDMPs. 
These two studies identified different limits, which stem from different choices of the velocity distribution $\mu^d$.
In the work of \cite{Bierkens+2022}, $\mu^d$ is the uniform distribution on the unit sphere $S^{d-1}$ in $\R^d$, enabling a direct comparison between BPS and the Zig-Zag sampler (ZZS) \citep{Bierkens+2019}. With this choice $\mu^d=\Unif(S^{d-1})$, the coordinate processes of BPS have been proven to possess OU limits under an $O(d)$ time scaling. The authors also considered the negative log-density process. For both processes, they concluded an $O(d^2)$ computational complexity to generate approximately independent samples. On the other hand, in \cite{Deligiannidis+2021}, $\mu^d$ is set to the $d$-dimensional standard Gaussian distribution. Consequently, without time scaling, the finite dimensional coordinate processes of BPS have been proven to converge to the Randomized Hamiltonian Monte Carlo (RHMC) processes \citep{Bou-Rabee-Sanz-Serna2017}. In this way, they were able to theoretically validate the $O(d^{3/2})$ computational complexity for the one-dimensional marginal distribution observed in the numerical experiments of \cite{Bouchard-Cote+2018}. However, for the negative log-density process, the scaling differs and appears less favorable than in the case $\mu^d=\Unif(S^{d-1})$; see \cite{Deligiannidis+2021,Bierkens+2022}. In both cases, the negative log-density constitutes one of the slowest mixing quantities.

For other PDMP scaling regimes, the authors of \cite{Bierkens+2025} consider an anisotropic limit, where the covariance structure of a two-dimensional Gaussian distribution factorizes into two increasingly different length scales. In \cite{Agrawal+2025ZZ}, the authors consider the large sample limit $n\to\infty$ to compare different subsampling schemes of ZZS. In \cite{Agrawal+2025Fluid}, a central aim is to investigate the fluid limit as the initial starting point moves infinitely away from the mode to quantify how quickly different PDMP samplers return to the high-density region. In doing so, they consider a family of targets of the form $\pi^\epsilon(x)\propto e^{-U(x)/\epsilon}$ for $\epsilon>0$ and study the limit of $\epsilon\to0$. In their analysis, under appropriate assumptions, FECMC and the closely related Coordinate Sampler \citep{Wu-Robert2020} are proved to perform outstandingly better than the other methods, including BPS and ZZS, achieving $O(1)$ computational complexity in the limit.

\subsection{Setting and Notation}

In this work, we focus on the radial momentum and the negative log-density processes. Our choice is motivated by the observation that the mixing of these two quantities poses a more difficult yet statistically relevant challenge for PDMP samplers; see also \cite{Terenin-Thorngren2018,Sherlock-Thiery2021,Deligiannidis+2021}. Our setting is therefore more suitable for approaching the worst-case computational complexity of these algorithms in a way that is closely aligned with the $L^2$-convergence rate results in \cite{Andrieu+2021,Lu-Wang2022,Lu-Wang2022ZZ} and the total variation convergence results in \cite{Deligiannidis+2019,Durmus+2020,Vasdekis-Roberts2022,Roberts-Rosenthal2023}. We set the velocity distribution $\mu^d$ to be the uniform distribution on the unit sphere $S^{d-1}$ in $\R^d$, i.e., $\mu^d=\Unif(S^{d-1})$, in order to match the analysis of BPS in \cite{Bierkens+2022}. This choice is also made in the experiments of \cite{Michel+2020}, presumably to enable direct comparison with the original event-chain Monte Carlo algorithm of \cite{Bernard+2009,Michel+2014}.

Throughout our theoretical analysis in Sections~\ref{sec-ScalingAnalysis} and \ref{sec-Diffusivity}, the target distribution $\pi^d$ is assumed to be the $d$-dimensional standard Gaussian distribution
\begin{equation}\label{eq-pi-d}
    \pi^d(x)\coloneq \frac{1}{(2\pi)^{d/2}}\exp\paren{-U^d(x)},\qquad U^d(x)\coloneq \frac{\abs{x}^2}{2}=\frac{x_1^2+\cdots+x_d^2}{2},
\end{equation}
where the negative log-density function $U^d(x)$ will also be called the \textit{potential} function hereafter. Observe that $U^d$ is a function of the radial distance $\abs{x}$, and therefore $\pi^d$ is a spherically symmetric distribution. For such targets, the potential process captures mixing in the radial direction. Our proof framework applies to more general product-form distributions, and we expect the OU limits in Theorems~\ref{thm-Y} and \ref{thm-Y-positive-rho} to remain valid under appropriate integrability conditions. We do not pursue sharp sufficient conditions here. A systematic treatment of such conditions, as well as extensions to dependent targets, will be an interesting direction for future work, as the numerical experiments in Section~\ref{subsec-experiment-gamma} suggest that this diffusive limit may persist even under weak dependence. Instead, by taking advantage of the Gaussianity, we obtain closed-form expressions for $\sigma$ in Theorem~\ref{thm-formula-for-sigma}, which lead to a rigorous asymptotic efficiency ordering between FECMC and BPS. For general targets, the diffusion coefficient $\sigma$ typically needs to be evaluated numerically; see Section~\ref{sec-Experiment} for empirical estimates.

\subsection{Piecewise Deterministic Monte Carlo}

PDMPs form a class of continuous-time Markov processes that complement diffusion processes \citep{Davis1984,Davis1993}.
Most PDMPs that appear in modern Monte Carlo methodology have linear dynamics between random events, and their extended generators take the form
\begin{equation}\label{eq-PDMP-generator}
    Lf(z)=(v|\nabla_xf(z))+\lambda(z)\int_{\R^{d}}\Paren{f(x,v')-f(x,v)}Q(z,\d v'),\qquad z=(x,v)\in\R^d\times\R^d,
\end{equation}
for test functions $f$ in an appropriate domain $\cD(L)$. Here, $\nabla_xf=(\partial_{x_1}f,\cdots,\partial_{x_d}f)^\top$ denotes the gradient with respect to $x$, $\lambda:\R^{2d}\to\R_+\coloneq \cointerval{0,\infty}$ is a continuous function called the \textit{rate} (or intensity) function, and $Q:\R^{2d}\times\cB(\R^{d})\to[0,1]$ is a Markov kernel describing the jumps. 
$(\cdot|\cdot)$ denotes the inner product throughout.

To turn the PDMP \eqref{eq-PDMP-generator} into a Monte Carlo sampler targeting a probability density $\pi$ defined on $\R^d$, a velocity distribution $\mu$ on $\R^d$ must be specified first. Then, typically, the PDMP with the product target $\pi\otimes\mu$ is simulated. There are many choices for $\lambda$ and $Q$ to ensure convergence to $\pi\otimes\mu$. A common choice for the rate function is
\begin{equation}\label{eq-lambda}
    \lambda(x,v)=(v|\nabla U(x))_++\rho,\qquad(x,v)\in\R^{2d},\rho\ge0,
\end{equation}
where $U(x)=-\log\pi(x)$ is the potential, $\rho$ is called the \textit{refreshment rate}, and the notation $(x)_+\coloneq \max(x,0)$ stands for the non-negative part of a real number $x$. Paired with this rate function, the jump is typically given by
\begin{equation*}
    Q(z,B)=\frac{(v|\nabla U(x))_+}{\lambda(z)}Q_V(z,B)+\frac{\rho}{\lambda(z)}\mu(B),\qquad B\in\mathcal{B}(\R^d),
\end{equation*}
where $Q_V$ is called a \textit{velocity jump kernel}. Note that this choice of $Q$ only changes the velocity $v$, and PDMPs with such a kernel $Q$ are called \textit{velocity-jump PDMPs} \citep{Othmer+1988,Monmarche+2020,Durmus+2021}.

For example, BPS employs a deterministic jump called \textit{reflection}, which is defined by $Q_V(z,\cdot)=\delta_{\phi(z)}$, where $\delta_v$ is a Delta measure with its point mass on $v\in\R^d$, and $\phi:\R^{2d}\to\R^d$ is given by
\begin{equation}\label{eq-phi}
\phi(x,v)=\begin{cases}
    v-2v^\tpar,&\nabla U(x)\ne0,\\
    v,&\nabla U(x)=0,
\end{cases}
\end{equation}
where $v^\tpar\coloneq (n(x)|v)n(x)$ denotes the projection of $v$ onto $n(x)\coloneq \nabla U(x)/\abs{\nabla U(x)}$, which is defined when $\nabla U(x)\ne0$. We will introduce the FECMC jump kernel in the following subsection \ref{subsec-FECMC}. For other design strategies for $\lambda$ and $Q$, consult \cite{Fearnhead+2018,Vanetti2019,Wu-Robert2020}. Notably, \cite{Bou-Rabee-Sanz-Serna2017,Bierkens+2020,Kleppe2022} employs nonlinear dynamics as its deterministic flow.

Under the above setting, a general condition ensuring invariance has been derived in \cite{Michel+2020}, which we summarize in the following proposition.

\begin{proposition}
    Assume that the potential $U(x)$ is continuously differentiable and $C_c^1(\R^d)$ is a core of the generator $L$ given in \eqref{eq-PDMP-generator}. If for every $C^1$-function $f$ with compact support,
    \begin{equation}\label{eq-FECMC-invariance}
        \int_{\R^d}(v|\nabla U(x))_+Qf(x,v)\,\mu(\d v)=\int_{\R^d}(v|\nabla U(x))_-f(x,v)\,\mu(\d v),\qquad \pi\text{-a.s.}\; x\in\R^d,
    \end{equation}
    where $(x)_-\coloneq \min(x,0)$, then the product distribution $\pi\otimes\mu$ is the invariant distribution of the PDMP corresponding to the generator $L$.
\end{proposition}

Note that invariance alone does not necessarily imply convergence. 
To ensure ergodicity, BPS requires a positive refreshment rate $\rho>0$ in \eqref{eq-lambda}; see \cite{Bouchard-Cote+2018}. Choosing $\rho$ too large injects uninformative noise, increases computational overhead, and leads to random-walk behavior.
FECMC was proposed in \cite{Michel+2020} to resolve this issue by employing a stochastic velocity jump $Q_V$. This FECMC jump kernel $Q_V$ enables the choice of $\rho=0$ without losing its ergodicity and therefore frees the algorithm from this hyperparameter. 
In the following subsection, we detail the design of $Q_V$.

\subsection{The FECMC Algorithm}\label{subsec-FECMC}

For FECMC, the velocity distribution $\mu$ is assumed to be spherically symmetric. Consider an event triggered at the point $z=(x,v)$. The FECMC jump kernel $Q_V$ operates on an orthogonal decomposition of $v$ with respect to the gradient direction $n(x)=\nabla U(x)/\abs{\nabla U(x)}$ at $x$
\[v=v^\tpar+v^\perp,\qquad v^\tpar\coloneq (v|n(x))n(x),\]
where we suppressed the dependence on $x$ in the notation. The new velocity $v_{\text{new}}\sim Q_V(z,\cdot)$ is also constructed as $v_{\text{new}}=v_{\text{new}}^\tpar+v_{\text{new}}^\perp$, where the two components are specified in turn as follows.

\begin{enumerate}[{Step}1]
    \item The component $v_{\text{new}}^\tpar$, the projection onto the gradient direction $n(x)$, is set by $v_{\text{new}}^\tpar\coloneq -wn(x)$, where the length $w=\abs{v_{\text{new}}^\tpar}$ is completely refreshed from a distribution $q^\tpar$. $q^\tpar$ is specified to ensure the invariance condition \eqref{eq-FECMC-invariance}.
    Specifically, this density $q^\tpar$ needs to take the form of
    \begin{equation}\label{eq-q-tpar-general}
        q^\tpar(w)\;\propto\; w\mu_1(w)1_{[0,\infty)}(w),\qquad w\in\R,
    \end{equation}
    where $\mu_1$ is the marginal density of the first component of $\mu$ \citep{Michel+2020}.
    Note that this component $v_{\text{new}}^\tpar$ does not depend on the previous velocity $v$.
    \item The perpendicular component $v_{\text{new}}^\perp$ is `switched' by first randomly choosing an orthogonal matrix $A\in\R^{d\times d}$ and then transforming the previous perpendicular component $v^\perp$ by $Av^\perp$. After renormalization, $v_{\text{new}}^\perp$ is specified as
    \[v_{\text{new}}^\perp=\frac{1}{\sqrt{1-w^2}}Av^\perp,\qquad A\sim\nu^x,\]
    where $\nu^x$ is a certain probability distribution defined on the space of orthogonal matrices.
\end{enumerate}

Several choices for the probability distribution $\nu^x$ are considered in \cite{Michel+2020}. Among them, the \textit{naïve orthogonal switch} performs consistently best in the reported numerical results. This approach constructs the matrix $A$ by $A\coloneq I_d-(e_1-e_2)(e_1-e_2)^\top$, where $I_d$ is the $d$-dimensional identity matrix, and $e_1$ and $e_2$ consist of a pair of orthogonal unit vectors chosen uniformly at random from the unit sphere in
\[n(x)^\perp\coloneq \{u\in\R^d\mid (u|n(x))=0\}.\]
A hyperparameter $p\in(0,1)$ can be further incorporated to control the frequency of switching, for example, by using $A$ with a probability of $p$ and $I_d$ in the other cases. Consult \citet[Section~3]{Michel+2020} for other choices for $\nu^x$ and further implementation strategies.

In our theoretical analysis, we do not specify $\nu^x$ since the changes caused by $\nu^x$ do not affect the radial direction due to the spherical symmetry of the standard Gaussian distribution. For the numerical experiments in Section~\ref{sec-Experiment}, however, we use orthogonal switches with a probability $p=0.05$, following the practice of \citet[Section~4]{Michel+2020}. For our choice of $\mu^d=\Unif(S^{d-1})$, the distribution $q^{\tpar,d}$ defined in \eqref{eq-q-tpar-general} reduces to
\begin{equation}\label{eq-q-tpar}
    q^{\tpar,d}(w)=(d-1)w(1-w^2)^{\frac{d-3}{2}}1_{(0,1)}(w),\qquad w\in\R.
\end{equation}
The asymptotic properties of this distribution play a crucial role in establishing our limit theorems; hence, they are studied in Appendix~\ref{sec-appendix-B}.

\section{Scaling Analysis of FECMC}\label{sec-ScalingAnalysis}

Throughout, we study the FECMC process $Z^d_t=(X^d_t,V^d_t)$ under the following assumption.
\begin{assumption}\label{assumption-FECMC}
    Assume that $Z^d$ admits the product-form invariant distribution $\pi^d\otimes\mu^d$, where $\pi^d=N_d(0,I_d)$ is the $d$-dimensional standard Gaussian distribution and $\mu^d=\Unif(S^{d-1})$ is the uniform distribution on the unit sphere $S^{d-1}\subset\R^d$; see also \eqref{eq-pi-d}. Additionally, we assume a stationary start: $Z_0^d\sim\pi^d\otimes\mu^d$.
\end{assumption}
Under the high-dimensional asymptotic regime $d\to\infty$, we study the limiting behavior of the following two processes: the \textit{radial momentum process}
\begin{equation}\label{eq-Rd}
    R^d_t\coloneq (X^d_t|V_t^d)=\sum_{i=1}^dX_t^{d,i}V_t^{d,i},\qquad t\ge0,
\end{equation}
and the scaled \textit{potential process}
\begin{equation}\label{eq-Yd}
    Y^d_t\coloneq \frac{2U^d(X^d_{dt})-d}{\sqrt{d}}=\frac{\abs{X^d_{dt}}^2}{\sqrt{d}}-\sqrt{d},\qquad t\ge0.
\end{equation}
Note that on the right-hand side of \eqref{eq-Yd}, $X^d$ is accelerated by a factor of $O(d)$ to derive a non-degenerate scaling limit, while $R^d$ remains on the original scale.
Both processes are hereafter viewed as random variables taking values in the Skorokhod space.
As such, we first derive a PDMP scaling limit for $R^d$, together with a few properties of this limit, in Section~\ref{subsec-radial-momentum}.
Observe that $R^d$ is the time derivative of $Y^d$. This result is then used to establish the diffusion limit of $Y^d$ in Section~\ref{subsec-Potential}.
Although both limiting processes, which we denote by $Y$ and $R$, are Markov, for any fixed $d\ge 1$, neither $X^d$ nor $R^d$ is Markov with respect to its natural filtration.
This poses a major challenge in deriving the limit theorems. This was partially resolved within a semimartingale framework in the proof of Theorem~2.10 in \cite{Bierkens+2022}. However, the argument crucially relies on the regeneration structure in $R^d$ induced by global refreshment. Therefore, further techniques are required to address the case where global refreshment is absent, i.e., $\rho=0$. This issue is specific to FECMC. For BPS, the limit theorem does not hold when $\rho=0$ due to the loss of ergodicity.
We outline a unifying framework that subsumes this boundary case in Section~\ref{subsec-proof-strategy}, while deferring all proofs to the appendices. Appendix~\ref{sec-appendix-A} contains the proof and the required auxiliary propositions for our main Theorem~\ref{thm-Y}. It uses results on the jump structure in Appendix~\ref{sec-appendix-B} and the limit theorem for $Y^d$ in Appendix~\ref{sec-appendix-C}.

\subsection{Radial Momentum}\label{subsec-radial-momentum}

The standard Rayleigh distribution (equivalently, the $\chi$ distribution with two degrees of freedom), denoted by $\chi(2)$, has the density $xe^{-x^2/2}$ on $x>0$; see, e.g., \citet[Chapter~39]{Forbes+2010}.

\begin{proposition}[Limit of the Radial Momentum Process]\label{prop-R}
    Suppose Assumption~\ref{assumption-FECMC} holds and set the global refreshment rate to zero ($\rho=0$).
    As $d\to\infty$, the radial momentum process $R^d$ \eqref{eq-Rd} converges weakly to a PDMP $R^F$, whose (extended) generator $L^F$ is given by
    \begin{equation}\label{eq-L_F}
        L^Ff(x)=f'(x)+x_+\Paren{\E[f(-\tau)]-f(x)},
    \end{equation}
    where $x_+\coloneq\min(x,0)$ and the expectation is taken with respect to $\tau\sim\chi(2)$, a standard Rayleigh random variable.
\end{proposition}

The superscript $F$ stands for FECMC. A corresponding radial momentum process $R^B$ and a generator $L^B$ for BPS will be defined subsequently in Remark~\ref{remark-sigma-BPS}.

\begin{corollary}[Number of Velocity Jumps]\label{cor-number-of-events}
    The number of velocity jumps of $Z^d$ over a fixed time interval $\ocinterval{0,T}$ satisfies
    \[\E\Square{\sum_{0\le t\le T}1_{\Brace{\Delta V_t^d\ne0}}}=\frac{T}{\sqrt{2\pi}},\qquad T>0,\]
    for all $d=1,2,\cdots$.
\end{corollary}

\begin{remark}[Asymptotic Ratio of Mean Jump Frequencies]\label{remark-jump-frequency}
    In a typical implementation of PDMP algorithms, the number of jumps approximately corresponds to the number of gradient evaluations. Thus, the number of jumps is widely used as a proxy for computational complexity; see \cite{Bierkens+2019,Krauth2021}. For FECMC, this quantity is asymptotically strictly smaller than that of BPS (cf. Corollary~2.9 of \citealt{Bierkens+2022}), partly because the (optimal) FECMC process does not employ any global refreshment move (corresponding to $\rho=0$). When the asymptotically optimal choice $\rho^*$ (see Theorem~\ref{thm-formula-for-sigma} below) is employed in BPS, the expected number of jumps for FECMC is approximately
    \[\frac{\frac{1}{\sqrt{2\pi}}}{\frac{1}{\sqrt{2\pi}}+\rho^*}\approx0.218\cdots\]
    times that of BPS.
\end{remark}

\begin{remark}[Asymptotic Ratio of ESS per Event]
    Since the asymptotic variance of the ergodic average of a FECMC trajectory is inversely proportional to $\sigma^2_F$ (Theorem~\ref{thm-formula-for-sigma}), the effective sample size (ESS) is proportional to $\sigma^2_F$.
    Combining this with the estimate $\sigma^2_F/\sigma^2_B(\rho^*)\approx1.73\cdots$ in Remark~\ref{remark-optimal-sigma}, the ESS per event under FECMC is approximately $1.73/0.218=7.93\cdots$ times larger than that under BPS.
    In Section~\ref{subsec-experiment-d}, we additionally observe an approximate fifteen-fold gap in ESS per CPU time. This additional widening is explained by the per-event costs. The global refreshment move in BPS requires drawing a new $d$-dimensional velocity from $\mu^d$, followed by normalization.
\end{remark}

\begin{proposition}[Exponential Ergodicity of $R^F$]\label{prop-exponential-ergodicity}
    Let $\gamma$ denote the one-dimensional standard Gaussian distribution.
    The limiting radial momentum process $R^F$ is $\gamma$-invariant and exponentially ergodic. Specifically, there exists a function $V\in L^1(\gamma)$ and constants $c_1,c_2>0$ such that
    \[\|P^t(x,-)-\gamma\|_{\TV}\le c_1e^{-c_2t}V(x),\qquad t\ge0,\;x\in\R,\]
    where $P^t$ is the Markov transition kernel associated with $R^F$.
\end{proposition}

\subsection{Potential}\label{subsec-Potential}

Now we formulate our main result, first for the case $\rho=0$. The case $\rho>0$ is covered in Theorem~\ref{thm-Y-positive-rho}.

\begin{theorem}[Scaling Limit of the Potential Process when $\rho=0$]\label{thm-Y}
    Suppose Assumption~\ref{assumption-FECMC} holds and set the global refreshment rate to zero ($\rho=0$).
    As $d\to\infty$, the scaled potential process $Y^d$ converges in law to an Ornstein--Uhlenbeck (OU) process that solves an SDE
    \begin{equation}\label{eq-Y-FECMC}
        \d Y_t=-\frac{\sigma_F^2}{4}Y_t\,\dt+\sigma_F\,\d B_t,
    \end{equation}
    where $\sigma_F^2\coloneq \sqrt{32/\pi}$ and $(B_t)_{t\ge0}$ is a one-dimensional standard Brownian motion.
\end{theorem}


\begin{remark}[Comparison with BPS]\label{remark-sigma-BPS}
    The functional form of \eqref{eq-Y-FECMC} coincides with that of BPS, as reported in Theorem 2.10 of \cite{Bierkens+2022}. The only difference is that the coefficient $\sigma_F^2$ is replaced by
    \begin{equation}\label{eq-sigma-B-primitive}
        \sigma_B^2(\rho)\coloneq 8\int^\infty_0e^{-\rho t}K^B(t,0)\,\dt,\qquad\rho>0,
    \end{equation}
    where $K^B$ is the covariance function of the BPS limiting radial momentum process $R^B$. The generator $L^B$ associated with $R^B$ reads
    \begin{equation}\label{eq-L_B}
        L^Bf(x)=f'(x)+x_+\Paren{f(-x)-f(x)}.
    \end{equation}
    Therefore, the limiting process of FECMC is a time-changed version of BPS's limiting potential process through $t\mapsto (\sigma^2_F/\sigma^2_B(\rho))t$. For the optimal choice $\rho^*$ for BPS, the Monte Carlo estimate reported in \cite{Bierkens+2022} gives $\sigma^2_F/\sigma^2_B(\rho)\approx1.77$, indicating that the limiting FECMC potential dynamics are faster than those of BPS. Another approach for comparison may be formulated in terms of the \textit{speed measure}, as in \cite{Roberts+1997}, which can be defined for a broader class of diffusions \cite[Section VII.3]{Revuz-Yor1999}. For OU processes, the speed measure is inversely proportional to the drift coefficient. Note that the process is faster when the speed measure is smaller. Moreover, $t\mapsto\sigma_F^2t,\sigma_B^2t$ are the quadratic variations of these limiting processes. Larger values of $\sigma_F,\sigma_B$ are therefore favorable also from the viewpoint of quadratic variation, a continuous-time analogue for expected squared jump distance (ESJD) in discrete-time settings \citep{Sherlock2006,Sherlock-Roberts2009,Beskos+2013}.
\end{remark}

To facilitate a rigorous comparison of $\sigma_F$ and $\sigma_B$, we introduce the positive global refreshment rate $\rho>0$ into FECMC, making $\sigma_F$ a function of $\rho>0$, as is the case with the BPS diffusion coefficient $\sigma_B$ in \eqref{eq-sigma-B-primitive}. FECMC with global refreshment ($\rho>0$) is contrary to the recommendation of \cite{Michel+2020}. Our purpose here is twofold: (1) to theoretically confirm that the choice $\rho=0$ is indeed optimal for FECMC, and (2) to establish analytically that $\sigma_B(\rho)<\sigma_F$ holds for all $\rho>0$, which will be carried out in Section~\ref{sec-Diffusivity}.

\begin{theorem}[Scaling Limit of the Potential Process when $\rho>0$]\label{thm-Y-positive-rho}
    Suppose Assumption~\ref{assumption-FECMC} holds and the global refreshment rate is positive ($\rho>0$).
    As $d\to\infty$, the scaled potential process of FECMC converges in law to the OU process that solves the SDE \eqref{eq-Y-FECMC} with $\sigma_F$ replaced by
    \begin{equation}\label{eq-sigma-F-primitive}
    \sigma_F^2(\rho)\coloneq 8\int^\infty_0e^{-\rho t}K^F(t,0)\,\dt
    \end{equation}
    where $K^F$ is the covariance function of the FECMC limiting radial momentum process $R^F$.
\end{theorem}

\begin{remark}[Continuity at $\rho=0$]
    One might expect that Theorem~\ref{thm-Y} should be covered by taking the limit $\rho\searrow0$, i.e., $\sigma_F^2(\rho)\to\sigma_F^2$. However, this convergence turns out to be very delicate.
    We establish this result in Corollary~\ref{cor-continuity-at-0} below.
    By contrast, for BPS, this limit as $\rho\to0$ turns out to be zero, reflecting the failure of the limit theorem.
\end{remark}

The above theorem is closely related to Theorem~2.10 in \cite{Bierkens+2022}, which states a corresponding result for BPS. The argument in \cite{Bierkens+2022} relies on the regeneration structure induced by the global refreshment, combined with careful control of residual terms using Stein's method. However, Theorem~\ref{thm-Y} cannot be handled by the same approach and requires new techniques. We discuss our approach to addressing this in the following subsection.

\subsection{Proof Strategy}\label{subsec-proof-strategy}

A classical approach to weak convergence results such as Proposition~\ref{prop-R} and Theorems~\ref{thm-Y}, \ref{thm-Y-positive-rho} is to apply a Trotter--Kato type theorem, in which convergence is derived from the associated Markov semigroups or generators; see, e.g., \cite{Trotter1958,Yoshida1995}. However, in our setting, the processes $R^d$ and $Y^d$ are not Markov with respect to their natural filtrations. There are primarily two approaches to derive scaling limits for such non-Markovian processes. Let $A^1,A^2,\cdots$ be a sequence of such processes. One approach is to embed $A^d$ into a proper Markov process $\overline{A}^d$ on an augmented state space and view $A^d$ as its projection. Within the general framework of Section~4.8 in \cite{Ethier-Kurtz1986}, a suitable notion of generator convergence for $\overline{A}^d$ yields weak convergence of the marginal processes $A^d$ towards a possibly non-Markovian limiting process. Many of the previous works employ this strategy; see, e.g., \cite{Roberts+1997,Roberts-Rosenthal1998,Bierkens-Roberts2017,Yang+2020,Agrawal+2023} for applications of this approach. In doing so, one typically needs homogenization-type results in advance, such as Lemma~2.1 in \cite{Roberts+1997}, to verify the required generator convergence. In our case, however, the relevant homogenization-type result for limiting coefficients is obtained only after the limit theorem has been established (Corollary~\ref{cor-continuity-at-0}). This difference essentially arises from our choice of subject: we study the potential process instead of the finite-dimensional marginal coordinate processes, as in \cite{Roberts+1997}. Because the potential process is an additive functional, its diffusivity is a time integral involving the corresponding time-derivative process, known as the Green--Kubo formula; see Equation~\eqref{eq-sigma-to-resolvent}. Establishing that this quantity is finite, however, is generally delicate. A natural route is to derive an estimate for the resolvent operator $(\lambda-L)^{-1}$ as $\lambda\searrow0$, associated with the generator $L$, in the spirit of the Kipnis--Varadhan program \citep{Kipnis-Varadhan1986}, but this can be technically demanding; see also Section~\ref{sec-Conclusion}.

We therefore adopt a second approach to establish scaling limit theorems for the potential process and the radial momentum process. This approach yields a cleaner separation between the probabilistic derivation of the limit and the subsequent analytic reinterpretation. We view $A^d$ as a semimartingale and work with its predictable characteristics; see \citet[Definition~II.2.6]{Jacod-Shiryaev2003} and \citet[Definition~32.2]{Metivier1982} for a precise definition. The semimartingale characteristics are a triplet that can be regarded as a probabilistic generalization of the generator of a Markov process, as they determine the distribution of a semimartingale under suitable conditions. To establish convergence, we prove that the characteristics converge in an appropriate sense within the framework of \citet[Chapter~IX]{Jacod-Shiryaev2003}. PDMPs are particularly well suited for this semimartingale approach because their underlying point process structure naturally gives rise to semimartingale decompositions. This strategy has been applied to PDMPs in \cite{Bierkens+2022} and to MH-based methods in \cite{Kamatani2018,Kamatani2020}, and is closely related to martingale CLT approaches to function space MCMC methods \citep{Mattingly+2012,Pillai+2012,Pillai+2014,Ottobre+2016}.


\section{Diffusivity of FECMC and BPS}\label{sec-Diffusivity}

In this section, we derive analytic expressions for the diffusion coefficient $\sigma_F^2,\sigma_B^2$, which appeared in our limit theorems in Equations~\eqref{eq-sigma-F-primitive} and \eqref{eq-sigma-B-primitive} respectively. Perhaps surprisingly, the resulting formulae obtained in Theorem~\ref{thm-formula-for-sigma} are very explicit, being rational functions of the moment generating function of the standard Rayleigh distribution; see Section~\ref{subsec-formulae-for-sigma}. The values of $\sigma_F^2$ and $\sigma_B^2$ computed from Theorem~\ref{thm-formula-for-sigma} are plotted as functions of $\rho>0$ in Figure~\ref{fig-sigma}. As illustrated in Figure~\ref{fig-sigma}, these representations enable a rigorous comparison of asymptotic efficiency between FECMC and BPS. Additionally, we can conclude that any global refreshment introduced by setting $\rho>0$ will slow down the limiting process and thus deteriorate the efficiency of FECMC; see Section~\ref{subsec-optimal-FECMC}. All proofs are deferred to Appendix~\ref{sec-appendix-D}.

\subsection{Analytic Formulae for Diffusion Coefficients}\label{subsec-formulae-for-sigma}

$\sigma^2_F$ and $\sigma^2_B$ turn out to be determined solely through associated radial momentum processes $R^F,R^B$.
Although the corresponding two generators $L^F$ and $L^B$ differ only in one term, the mixing properties of $R^F$ associated with $L^F$ are significantly improved over those of $R^B$. To see this, we focus on the function
\begin{equation}\label{eq-f-to-resolvent}
    f^F(x;\rho)\coloneq \int^\infty_0e^{-\rho t}\E[R_t^F|R_0^F=x]\,\dt=(\rho-L^F)^{-1}\id(x),\qquad\rho>0,
\end{equation}
where $\id(x)=x$ is the identity function. $f^B$ is defined analogously for BPS. The inverse operator $(\rho-L^F)^{-1}$ that appears on the right-hand side of \eqref{eq-f-to-resolvent} is called \textit{the resolvent} associated with the generator $L^F$, and it is well-defined for each $\rho>0$; see \citet[Proposition~2.1 p.10]{Ethier-Kurtz1986} or \citet[Theorem~II.1.10]{Engel-Nagel2000}. Using this function $f^F$, the diffusion coefficient $\sigma_F$ is expressed as
\begin{equation}\label{eq-sigma-to-resolvent}
    \frac{\sigma_F^2(\rho)}{8}=\int^\infty_0e^{-\rho t}\E[R^F_0R^F_t]\,\dt
    =\E\SQuare{R^F_0f(R^F_0;\rho)},\qquad\rho>0.
\end{equation}
The first equality is simply a rewriting of \eqref{eq-sigma-F-primitive} and expresses the diffusion coefficient as an integrated autocorrelation function; it is known as the \textit{Green--Kubo formula} \citep{Kubo+1991,Pavliotis2010}. The second equality holds due to Fubini's theorem after taking the conditional expectation. Consequently, using Equation~\eqref{eq-sigma-to-resolvent}, we will obtain explicit formulae for $\sigma_F$ and $\sigma_B$ by solving the resolvent equation \eqref{eq-f-to-resolvent} for $f^F$. The result turns out to be a rational function of the moment generating function of the standard Rayleigh distribution. Before stating the results, we introduce the following notation:
\begin{equation}\label{eq-Omega}
    \Omega(\rho)\coloneq \sqrt{\frac{\pi}{2}}\rho\erfcx\paren{\frac{\rho}{\sqrt{2}}}=\rho e^{\frac{\rho^2}{2}}\int^\infty_\rho e^{-\frac{t^2}{2}}\,\dt,
\end{equation}
where $\erfcx(x)=\frac{2e^{x^2}}{\sqrt{\pi}}\int^\infty_xe^{-t^2}\,\dt$ is the \textit{exponentially scaled complementary error function}; see \cite{Cody1969,Oldham+2009}. This $\Omega$ is related to the moment generating function through the relationship $\E[e^{\rho\tau}]=1-\Omega(-\rho)$, where $\tau\sim\chi(2)$; see Appendix~\ref{sec-appendix-D}. More importantly, this function $\Omega$ is numerically stable and was used to produce Figure~\ref{fig-sigma}, thereby avoiding numerical overflows due to scale differences.

\begin{theorem}[Formulae for $\sigma_F,\sigma_B$]\label{thm-formula-for-sigma}
    Under the same assumptions as Theorem~\ref{thm-Y-positive-rho},
    \begin{equation}\label{eq-sigma-F}
        \sigma_F^2(\rho)=\sigma_F^2\Paren{1-\frac{\paren{\rho^2-\rho\sqrt{\frac{\pi}{2}}+\Omega(\rho)}^2}{\rho^4\Omega(\rho)(2-\Omega(\rho))}},
    \end{equation}
    where $\sigma_F^2=\sqrt{32/\pi}$, and
    \begin{equation}\label{eq-sigma-B}
    \sigma_B^2(\rho)=\frac{8}{\rho^4}\paren{\rho^3-\rho^2\sqrt{\frac{8}{\pi}}+\rho-\sqrt{\frac{8}{\pi}}\frac{\paren{(1+\rho^2)\Omega(\rho)-\rho^2}^2}{\Omega(2\rho)}}.
    \end{equation}
\end{theorem}

\begin{remark}[Optimizing the Obtained Expression]\label{remark-optimal-sigma}
    Now that we have analytic expressions, we are able to numerically optimize $\sigma_B^2$, for example, using the algorithm of \cite{Brent1971}. This function attains its maximum value $\sigma_B^2(\rho^*)=1.838\cdots$ at $\rho^*=1.423\cdots$. This is strictly smaller than $\sigma_F^2=\sqrt{32/\pi}\approx3.19\cdots$.
    These values will be compared to the results of numerical experiments in Section~\ref{sec-Experiment}. In particular, we will observe that the ratio $\sigma^2_F/\sigma^2_B(\rho^*)\approx1.73\cdots$ is fairly robust to the assumptions and is valid under various target distributions other than the standard Gaussian distributions.
\end{remark}

\subsection{Optimal Global Refreshment Rate for FECMC}\label{subsec-optimal-FECMC}

In general, establishing the convergence of this quantity as $\rho\searrow0$ is delicate; see \cite{Komorowski+2012}. However, given the analytic formulae in Theorem~\ref{thm-formula-for-sigma}, we can conclude that the integral converges and recover the value predicted by Theorem~\ref{thm-Y}.

\begin{corollary}[Continuity at $\rho=0$]\label{cor-continuity-at-0}
    As $\rho\searrow0$, we have
    \[\sigma_F^2(\rho)\to\sigma_F^2=\sqrt{\frac{32}{\pi}},\qquad \sigma_B^2(\rho)\to0.\]
\end{corollary}

Next, we prove that the function $\sigma_F^2$ is maximized at $\rho=0$, thereby confirming that the design goal of \cite{Michel+2020} is met. This is also evident from Figure~\ref{fig-sigma}.

\begin{corollary}[Monotonicity of $\sigma^2_F$]\label{cor-derivative-of-rho}
    \[\frac{\d\sigma_F^2(\rho)}{\d\rho}<0,\qquad\rho>0.\]
\end{corollary}

\subsection{A Fast Proxy for Asymptotic Variance Estimation in High Dimensions}\label{subsec-fast-proxy}

In Section~\ref{sec-ScalingAnalysis}, we showed that the potential process $Y^d$ admits a non-degenerate limit only after scaling time by a factor $d$, whereas the radial momentum process $R^d$ has a stable limit on the original timescale. In other words, the mixing of the potential becomes increasingly slow as $d$ grows, while $R^d$ does not. This separation of timescales is also reflected in the efficiency of asymptotic variance estimation for ergodic averages.

We consider two test functions that correspond to $Y^d$ and $R^d$, namely the scaled negative log-density and its time derivative:
\begin{equation}\label{eq-h-g}
    h(x)\coloneq \frac{U(x)-\E_\pi[U(X)]}{\sqrt{\Var_\pi[U(X)]}},\qquad g(x,v)\coloneq \frac{(v|\nabla U(x))-\E_{\pi\otimes\mu}[(V|\nabla U(X))]}{\sqrt{\Var_{\pi\otimes\mu}[(V|\nabla U(X))]}}.
\end{equation}
By construction, $h$ and $g$ have zero mean and unit variance under the stationary distribution.
Consequently, it suffices to study the estimation of $\E_\pi[h(X)]$ and $\E_{\pi\otimes\mu}[g(X,V)]$, both of which are zero.
For the standard Gaussian target, each function simplifies to:
\[h(x)=\frac{1}{\sqrt{2}}\frac{|x|^2-d}{\sqrt{d}},\qquad g(x,v)=(x|v).\]
In Section~\ref{sec-Experiment}, we estimate their means using the ergodic averages
\begin{equation}\label{eq-hath-hatg}
\wh{h}_T^d\coloneq \frac{1}{T}\int^T_0h(X_s^d)\,\d s,\qquad \wh{g}_T^d\coloneq \frac{1}{T}\int^T_0g(X_s^d,V_s^d)\,\d s.
\end{equation}
The following proposition provides the values of the asymptotic variances of $\wh{h}_T^d$ and $\wh{g}_T^d$, respectively. The proof is given in Appendix~\ref{sec-appendix-D}.
\begin{proposition}[Asymptotic Variance Formulae]\label{prop-asymptotic-MSE}
    Under Assumption~\ref{assumption-FECMC},
    \[\lim_{T\to\infty}\lim_{d\to\infty}T\Var[\wh{h}^d_{dT}]=\frac{8}{\sigma^2_F},\qquad\lim_{T\to\infty}\lim_{d\to\infty}T\Var[\wh{g}^d_{T}]=\frac{\sigma^2_F}{4}.\]
    Note that, for $h^d$, we evaluate the time average over a horizon of length $dT$ to produce non-degenerate variance, matching the diffusive time scaling of $Y^d$ in Section~\ref{sec-ScalingAnalysis}.
\end{proposition}

Proposition~\ref{prop-asymptotic-MSE} suggests a fast proxy for the asymptotic variance estimation of $\wh{h}^d$. A common approach is to apply a batch means (BM) estimator constructed from time averages of $h$ over successive short time intervals called batches; see \cite{Flegal-Jones2010,Liu+2022} and also \cite{Bierkens+2019} for application to PDMPs. Since the correlation time of $h(X_t^d)$ grows on the $O(d)$ timescale, stable BM estimation for $h$ requires the batch size to scale at least on the order of $d$; see the $dT$ time horizon for $\wh{h}^d$ in Proposition~\ref{prop-asymptotic-MSE}. In contrast, the same proposition indicates that $\sigma_F^2$ can also be inferred from $\wh{g}^d$, which uses only a time horizon on the order of $O(1)$. In this way, one may estimate $\sigma_F^2$, and thus the asymptotic variance for $h$, using roughly $1/d$ of the computational budget. This reciprocal relationship is derived from the property of the OU limit and is observed to persist beyond the standard Gaussian target. We empirically demonstrate the effectiveness of this approach for different targets in Section~\ref{subsec-experiment-BM}.

This BM estimator via the fast proxy $g$ equals the unbiased sample variance of the energy increments over batches, i.e., $U^d(X^d_{t_i})-U^d(X^d_{t_{i-1}})$. This value is computationally cheaper than the standard BM estimator for $h$, since computing $\wh{h}_T^d$ generally requires integrating $h(X_t^d)$ along the trajectory via numerical quadrature. However, there is a pitfall in applying our fast proxy trick. Note that changing the order of the limits yields
\[\lim_{d\to\infty}\lim_{T\to\infty}T\Var[\wh{g}^d_T]=0.\]
This reflects the fact that the asymptotic variance of the ergodic average $\hat{g}^d_T$ is zero, an immediate fact that follows from the telescoping property. Therefore, our proposed estimator differs fundamentally from the usual application of a BM-type estimator, which would yield a trivial estimator. We will demonstrate that an appropriate choice of batch size will indeed lead to an efficient estimator for $\sigma^2_F$ in Section~\ref{subsec-experiment-BM}. The appropriate batch size appears to be much smaller than usual practice \citep{Liu+2022}; see Section~\ref{subsec-experiment-BM}. Identifying appropriate regimes in which our estimator has consistency, as well as a systematic treatment beyond the standard Gaussian targets, will be a practically important future direction.

Our estimator is closely related in spirit to Einstein--Helfand-type estimators for transport coefficients in molecular dynamics (MD) \citep{Helfand1960,Viscardy+2007}. 
Asymptotic variance estimation of ergodic averages is also very important in Bayesian inference, especially in MCMC output analysis and the development of adaptive versions of the algorithm. In the former, asymptotic variance estimators are utilized to assess convergence and sample quality \citep{Jones+2006,Flegal-Gong2015,Gong-Flegal2016}, report Monte Carlo standard errors (MCSEs) \citep{Flegal+2008}, and construct confidence regions for posterior estimates \citep{Atchade2016,Robertson+2021}; see \cite{Vats+2020,Roy2020} for further details. In the latter, the estimator of the asymptotic variance (or proxies thereof) is optimized internally to tune algorithmic parameters on the fly \citep{Andrieu-Robert2001,Pasarica-Gelman2010,Wang+2025}.

\section{Experiment}\label{sec-Experiment}

In this section, we validate our theory by first probing dimensionality scaling in Section~\ref{subsec-experiment-d} and then checking robustness to the Gaussian assumption in Section~\ref{subsec-experiment-gamma}. We mainly consider the mean estimation of the function $h$, as well as $g$ in Section~\ref{subsec-experiment-BM}. Both of them are defined in \eqref{eq-h-g}. While our theory has focused on the diffusion coefficients $\sigma_F,\sigma_B$, we mainly report the effective sample size (ESS) in this section for the sake of interpretability. We define the ESS for $h$ as $\ESS(h)={T}/{\varsigma^2_h}$, where $T$ is the time horizon and $\varsigma^2_h$ is the asymptotic variance of the ergodic average $\wh{h}_T$, which is defined in \eqref{eq-hath-hatg}. Recall that the variance of $h$ under stationarity is $1$ by construction.

Both samplers are initialized in stationarity. The refreshment rate $\rho$ of BPS is set to 1.42, the asymptotically optimal choice; see Remark~\ref{remark-optimal-sigma} and \cite{Bierkens+2022}. For FECMC, an orthogonal switch is performed approximately every fifty events (i.e., $p=0.05$), following the practice in \cite{Michel+2020}. Both algorithms share the same implementation that uses the automated Poisson thinning technique of \cite{Andral-Kamatani2024}. We implement only the direction change step separately, keeping the rest of the code identical for a fair comparison. The code is available in the Supplementary Material \citep{Shiba2026Supplement}.


\subsection{Dimensionality Scaling: FECMC vs. BPS}\label{subsec-experiment-d}

We vary the dimension from $d=10$ to $d=320$ and compare the values predicted by our theory with experimental values. We consider the same test function $h$ defined in \eqref{eq-h-g}. When both algorithms are run over the time interval $[0,dT]$, Proposition~\ref{prop-asymptotic-MSE} predicts that the ESSs should be independent of $d$, and for the standard Gaussian targets, the values should be
\begin{equation}\label{eq-ESS}
\ESS_{\text{FECMC}}(h)=\frac{T\sigma^2_F}{8}=\frac{T}{\sqrt{2\pi}}\approx 39.8,\qquad \ESS_{\text{BPS}}(h)=\frac{T\sigma^2_B(\rho^*)}{8}\approx22.9,
\end{equation}
respectively; see Remark~\ref{remark-optimal-sigma}. We set $T=100$. We construct an estimator for the mean squared errors (MSEs) of $h$ by averaging the squared errors of the empirical averages over $R=1000$ independent runs:
\[\wh{\MSE}_T\coloneq \frac{1}{R}\sum_{r=1}^R\paren{\frac{1}{dT}\int^{dT}_0h(X^{r}_t)\,\dt}^2,\qquad h(x)=\frac{U(x)-\Var_\pi[U]}{\sqrt{\Var_\pi[U]}}.\]
Finally, we report the estimated ESS calculated by $\wh{\ESS}_T=1/{\wh{\MSE}_T}$, together with its associated BCa bootstrap confidence interval.

We consider three different targets $\pi_1,\pi_2,\pi_3$: (1) the standard Gaussian $\pi_1$, (2) the multivariate standard logistic distribution $\pi_2(x)=\prod_{i=1}^d\frac{e^{x_i}}{(1+e^{x_i})^2}$, and (3) an anisotropic Gaussian $\pi_3\propto\exp(-x^\top\Sigma^{-1}x/2)$, where $\Sigma_{ii}=1$ and $\Sigma_{ij}=\gamma$ for $i\ne j$. For the third target, we set $\gamma=1/2$ in this section; other choices are explored in the next subsection.

First, for the standard Gaussian target $\pi_1$, we confirm that the ESSs of both FECMC and BPS are independent of the dimension $d$ and closely match the theoretical values, with all the 95\% confidence intervals covering the theoretical values; see the left panel of Figure~\ref{fig-Exp1-1}. In particular, our theory provides a very good approximation even in moderate dimensions, such as $d=10$ and $20$. When computational cost is taken into account, the difference in performance becomes more pronounced; see the right panel of Figure~\ref{fig-Exp1-1}. This is because FECMC requires fewer event simulations; see Remark~\ref{remark-jump-frequency}. The estimated ESS per CPU second amounts to approximately a fifteen-fold difference. The inverse linear dependence on dimensionality reflects the $O(d)$ scaling of the time horizon.

For the second experiment, the multivariate logistic distribution $\pi_2$ presents two additional challenges: heavier tails and a lack of spherical symmetry, while remaining isotropic across the coordinates. We observe that these departures from isotropic Gaussianity do not substantially alter the overall behavior, except for a reduction in the ESSs for both FECMC and BPS; see the left panel of Figure~\ref{fig-Exp1-2}. However, although the absolute values of ESS have changed, the ESS ratio between FECMC and BPS still remains close to the theoretical value $\sigma^2_F/\sigma^2_B(\rho^*)\approx1.73\cdots$, which is derived under the standard Gaussian assumption; see Remark~\ref{remark-optimal-sigma}

In the third experiment, the samplers are confronted with strong anisotropy induced by correlations among the coordinates. Once again, we observe behavior similar to that of the standard Gaussian case; see the right panel of Figure~\ref{fig-Exp1-2}. However, there are some notable deviations. First, the ESSs are slightly increased, which may be explained by a reduction in the intrinsic dimension caused by the correlation structure. Second, the ESS ratio is larger than that of the isotropic case. In the next section, we will observe that this relative efficiency gain widens as the level of anisotropy increases. We will discuss possible explanations there.

Overall, these experiments confirm the relevance of our high-dimensional analysis to even moderate dimensions and support the robustness of both the theoretical scaling predictions and the efficiency ratio prediction $\sigma^2_F/\sigma^2_B(\rho^*)\approx1.73\cdots$ across a range of targets with increasing structural complexity.

\begin{figure}[htbp]\centering
    \includegraphics[width=0.45\textwidth]{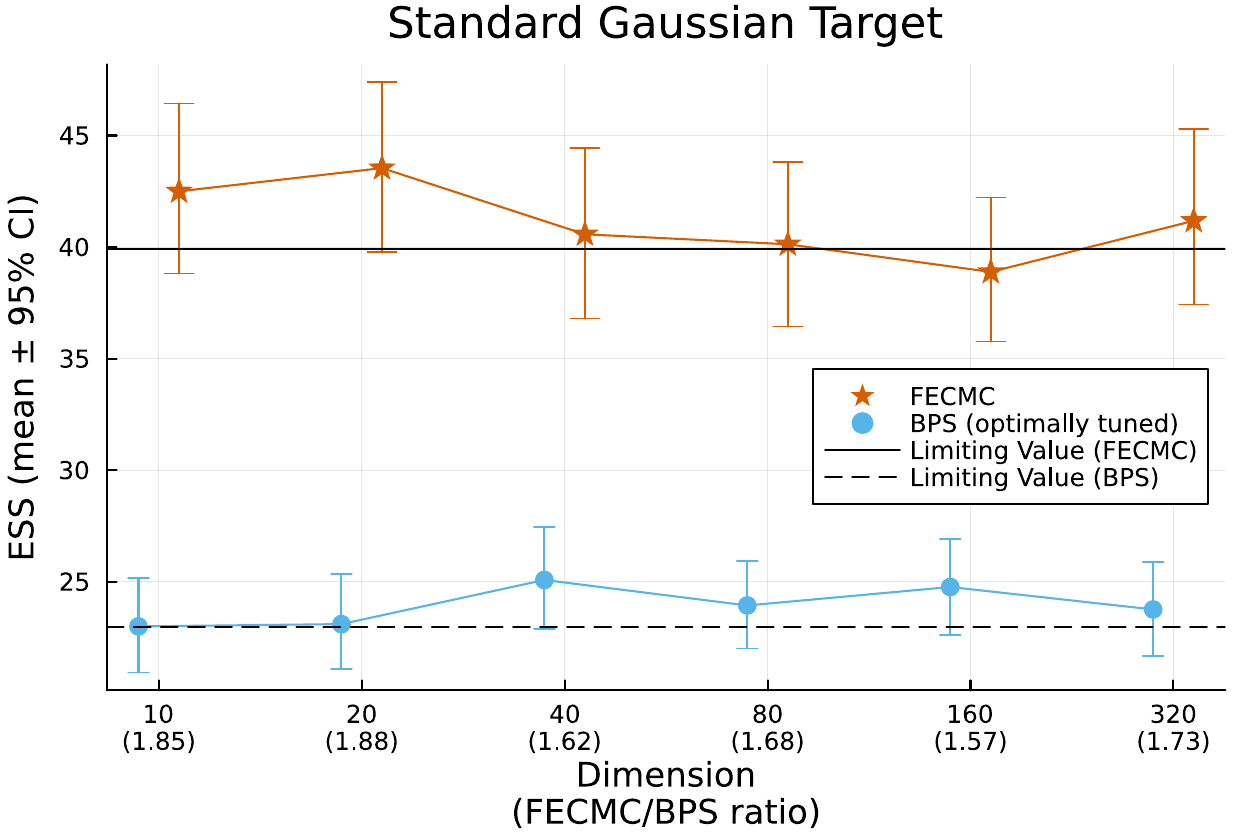}
    \includegraphics[width=0.45\textwidth]{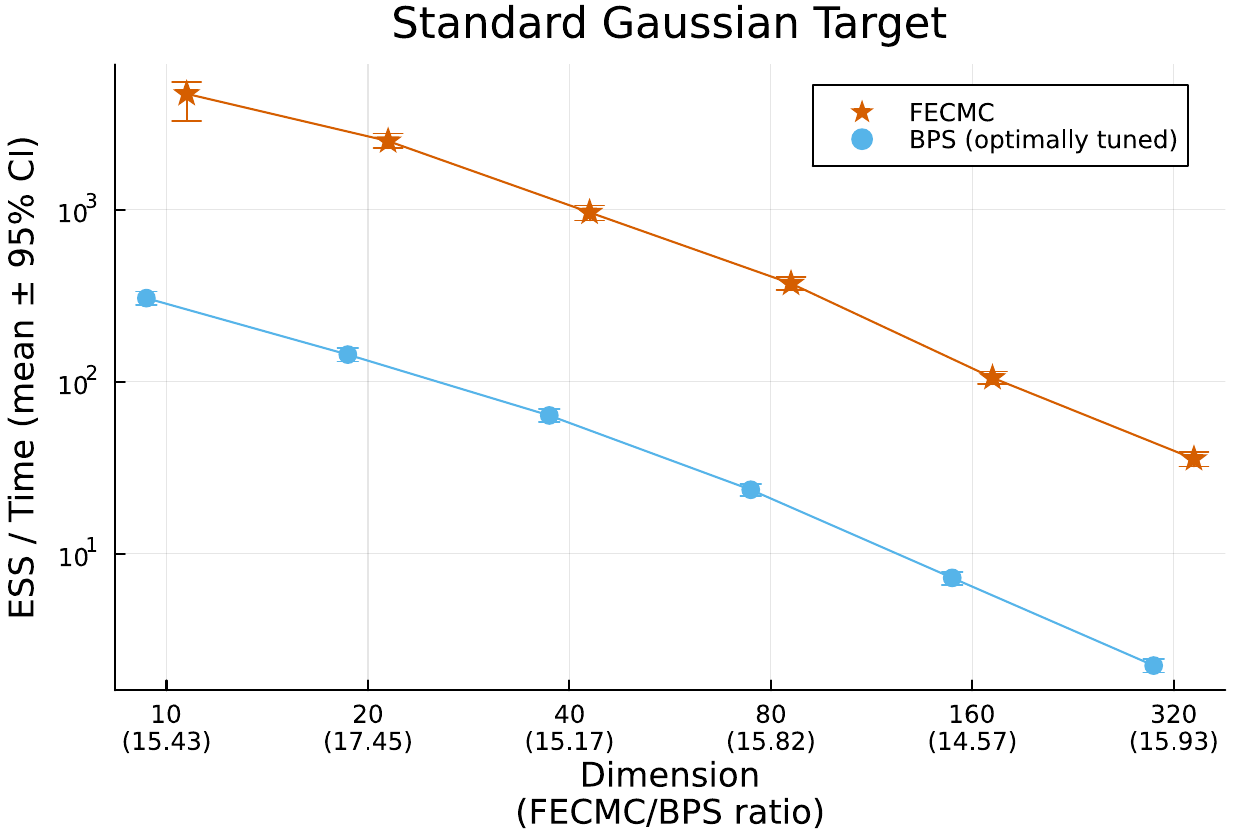}
    \caption{
    Estimated ESS (Left) and ESS per CPU second (Right), together with 95\% BCa bootstrap confidence intervals, against dimensionality.
    The parenthesized values below the $x$-ticks represent the ESS mean ratio of FECMC to BPS at each dimension.
    The estimator is given by $\wh{\ESS}_T=1/{\wh{\MSE}_T}$ for $T=100$.
    Both plots are based on 1000 independent runs of BPS and FECMC, targeting the standard Gaussian distribution $\pi_1(x)\propto\exp(-\abs{x}^2/2)$.
    The two black lines in the left plot represent the theoretical limiting values \eqref{eq-ESS} when $d,T\to\infty$ for FECMC and BPS respectively.
    }\label{fig-Exp1-1}
\end{figure}

\begin{figure}[htbp]\centering
    \includegraphics[width=0.45\textwidth]{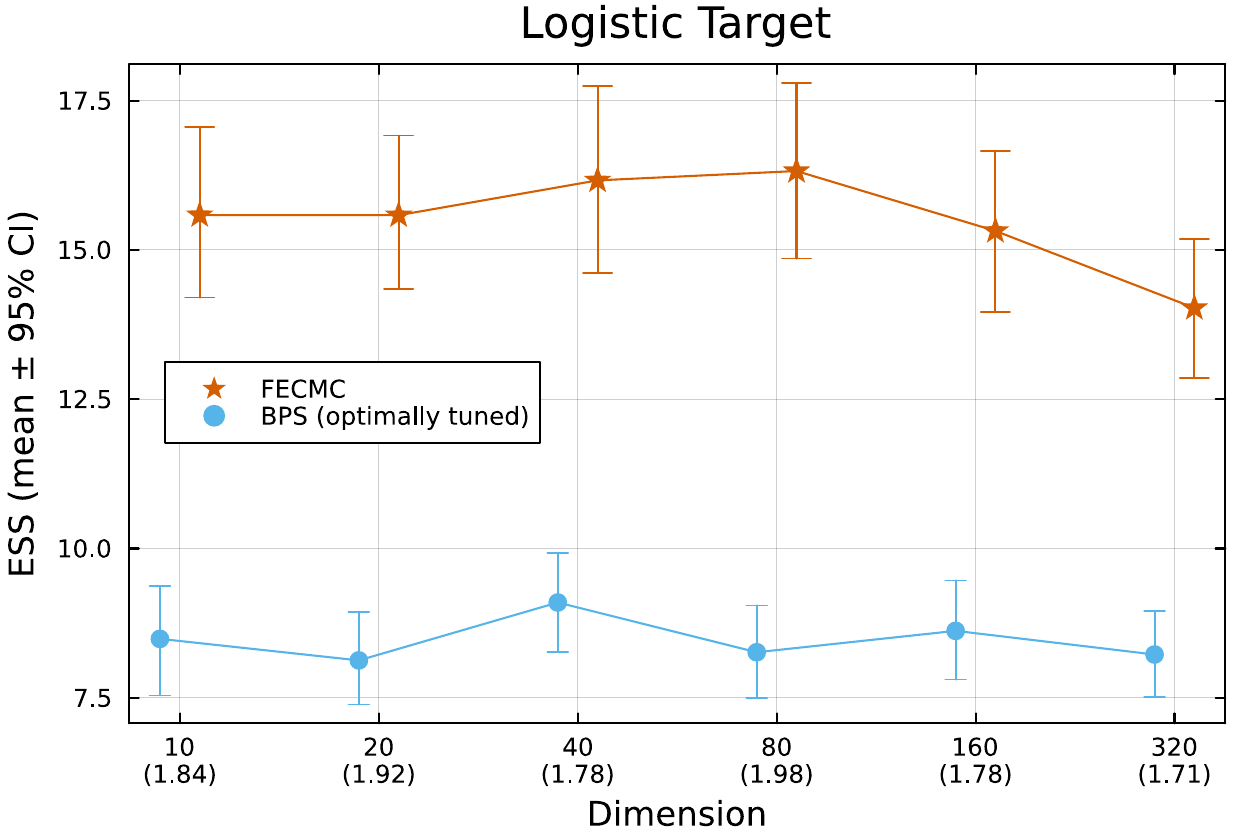}
    \includegraphics[width=0.45\textwidth]{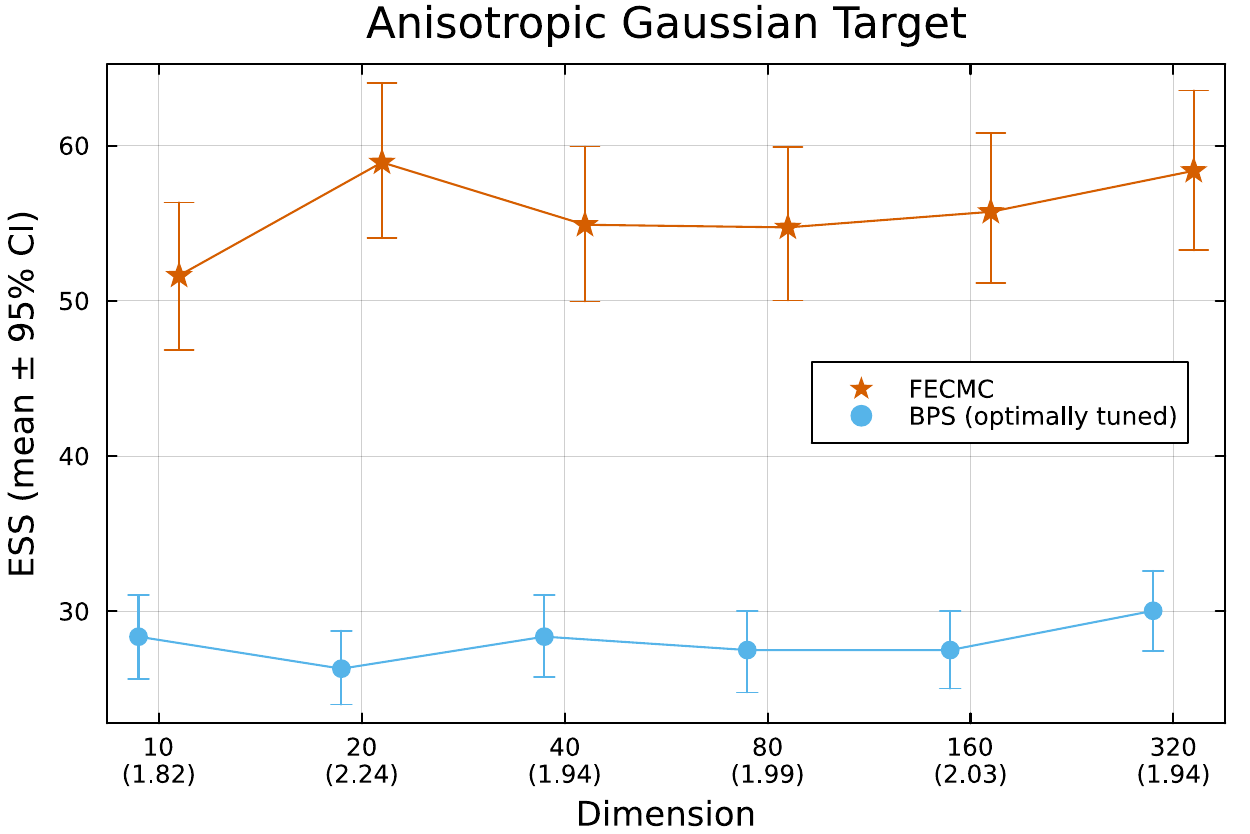}
    \caption{
    Estimated ESS against dimensionality, together with 95\% BCa bootstrap confidence intervals.
    The target distribution is either the i.i.d. logistic distribution $\pi_2(x)=\prod_{i=1}^d\frac{e^{x_i}}{(1+e^{x_i})^2}$ (left plot) or the anisotropic Gaussian distribution $\pi_3\propto\exp(-x^\top\Sigma^{-1}x/2)$, where $\Sigma_{ii}=1$ and $\Sigma_{ij}=0.5$ for $i\ne j$ (right plot). In both cases, each algorithm is run 1000 times independently with the time horizon $T=100$.
    }\label{fig-Exp1-2}
\end{figure}

\subsection{Robustness to Deviations from Isotropic Gaussianity}\label{subsec-experiment-gamma}

Fixing $d=100$, we perform two experiments that introduce progressively stronger deviations from isotropy and Gaussianity, respectively. In the first experiment, we revisit the anisotropic Gaussian setting $\pi_3(x)\propto\exp(-x^\top\Sigma^{-1}x/2)$ and vary the off-diagonal entries of the covariance matrix: $\Sigma_{ij}=\gamma\in\{0,0.1,0.2,\cdots,0.9\}$ for $i\ne j$. In the second experiment, we consider the spherically symmetric Student distribution $\pi_4(x)\propto(1+\abs{x}^2/\nu)^{-(d+\nu)/2}$ with varying degrees of freedom $\nu\in\{10,10^2,10^3,10^4\}$; see \cite{Fang+1990} for details on this distribution. Although the Student distribution has polynomial tails for any finite $\nu\ge1$, it converges to the standard Gaussian distribution as $\nu\to\infty$, thereby capturing varying degrees of tail heaviness. We will call the value of $\nu^{-1}$ tail heaviness.

The results are summarized in Figure~\ref{fig-Exp2}. The left edge in both plots corresponds to the Gaussian, or approximately Gaussian, target, where our theoretical predictions (indicated by dotted black lines) are within the 95\% confidence intervals. However, we also see that both experiments probe a broad spectrum of deviations from the standard Gaussian setting, including regimes in which the theoretical assumptions are substantially violated on the right edge. While a clear departure from the theoretical predictions is observed at extreme parameter values, the theory remains remarkably robust for moderate deviations, with both the ESS and the ESS ratio closely matching the predicted values. Interestingly, in the anisotropic Gaussian experiment (left panel of Figure~\ref{fig-Exp2}), when the correlation is very strong, such as $\gamma=0.9$, FECMC significantly outperforms BPS, amounting to a 3.6-fold gap in ESS. This phenomenon may warrant further investigation, particularly under alternative scaling regimes, such as those considered in \cite{Beskos+2018,Au+2023}. We hypothesize that this phenomenon is due to the irreversible nature of the fluid limit, in the sense of \cite{Fort+2008,Agrawal+2025Fluid}, of the one-dimensional linear projection FECMC process, which cannot be observed for BPS \citep{Bierkens+2022}.

\begin{figure}[htbp]\centering
    \includegraphics[width=0.45\textwidth]{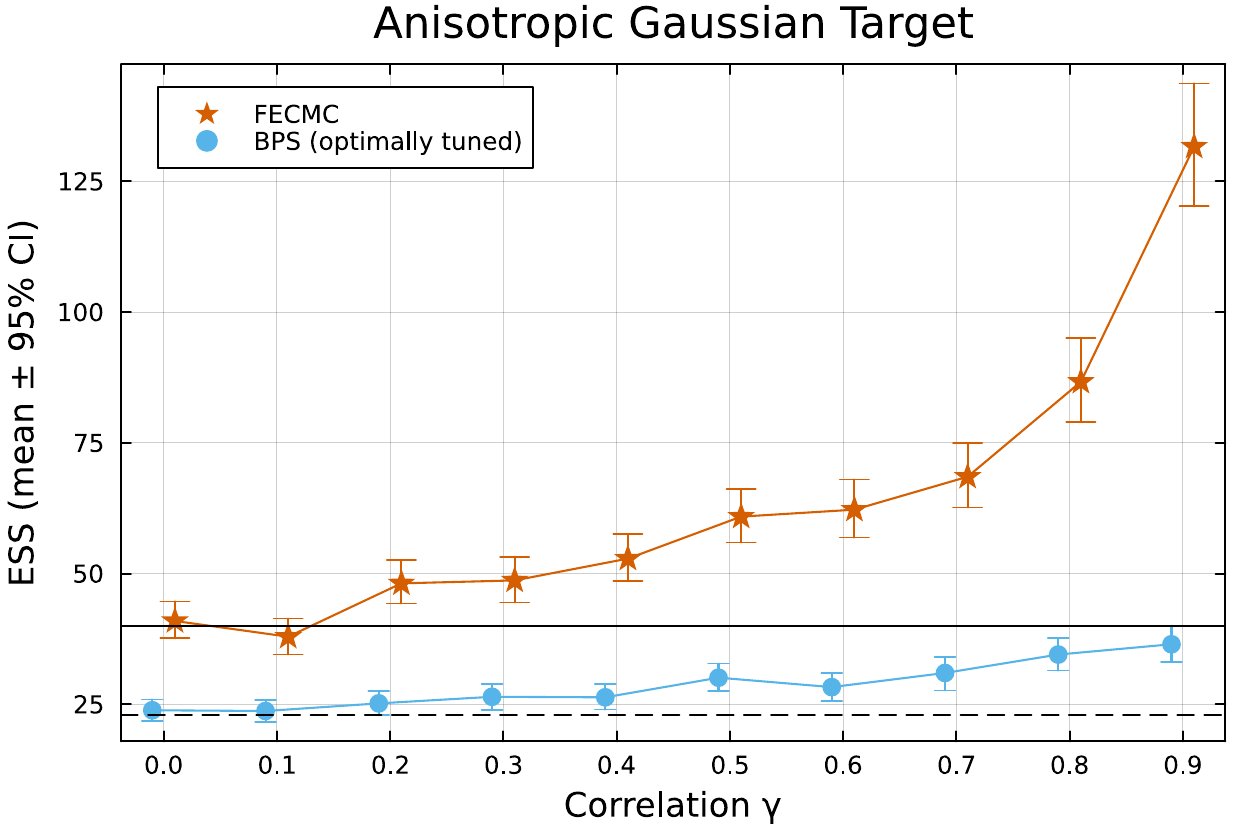}
    \includegraphics[width=0.45\textwidth]{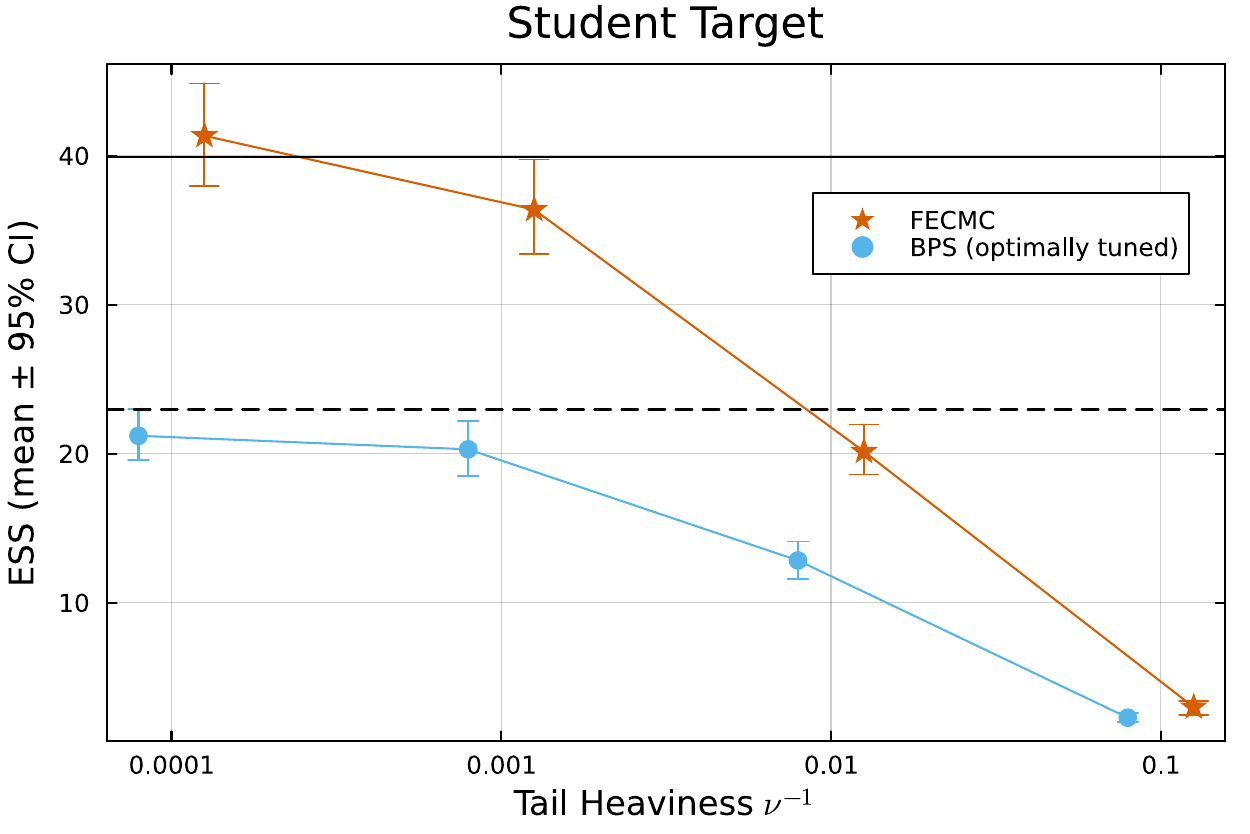}
    \caption{Estimated ESS against deviation parameters $\gamma,\nu^{-1}$, which quantify departures from isotropy and Gaussianity, together with 95\% BCa bootstrap confidence intervals.
    The target distribution is either the anisotropic Gaussian $\pi_3(x)\propto\exp(-x^\top\Sigma^{-1}x/2)$ with varying correlations $\Sigma_{ij}=\gamma\in\{0,0.1,0.2,\cdots,0.9\}$ (left plot) or the spherically symmetric Student distribution $\pi_4(x)\propto(1+\abs{x}^2/\nu)^{-(d+\nu)/2}$ with varying degrees of freedom $\nu\in\{10,10^2,10^3,10^4\}$ (right plot).
    The black lines denote theoretical limiting values derived under the standard Gaussian assumption in Eq.~\eqref{eq-ESS}
    In both cases, the dimension is $d=100$ and each algorithm is run 1000 times independently with the time horizon $T=100$.
    }\label{fig-Exp2}
\end{figure}

\subsection{A Fast Proxy for Asymptotic Variance Estimation}\label{subsec-experiment-BM}

In practice, the asymptotic variance of ergodic averages is estimated from a single trajectory, as it can be computationally expensive to run the algorithm several times independently. This estimate can also be aggregated among multiple chains to diagnose convergence \citep{Vats-Knudson2021}. A common choice for this task is the batch means (BM) estimator. For the test function $h$ and computational horizon $T$, it is given by:
\[\wh{\varsigma^2_h}=\frac{1}{B-1}\sum_{i=1}^{B}(Y_i-\ov{Y})^2,\qquad Y_i\coloneq \frac{1}{\sqrt{b}}\int^{ib}_{(i-1)b}h(X_t)\,\dt,\]
where $\ov{Y}$ is the sample mean of $\{Y_i\}_{i=1}^B$, and $b=T/B$ is the batch size; see also Section~2 in the supplement of \cite{Bierkens+2019}. This estimator is consistent as $T\to\infty$ under appropriate conditions \citep{Flegal-Jones2010}. However, this estimator has a notoriously severe bias-variance trade-off within a finite computational budget; therefore, variance reduction methods are actively pursued \citep{Flegal-Jones2010,Vats-Flegal2021,Liu+2022}.

In Section~\ref{subsec-fast-proxy}, we proposed a fast proxy trick: the time derivative of $h$ can yield a more efficient estimator for the asymptotic variance of $\wh{h}$. In this section, we empirically demonstrate that the BM estimator associated with $g$, denoted by $\wh{\varsigma^2_g}$, can be more efficient in terms of MSE than $\wh{\varsigma^2_h}$ in high dimensions. We consider two target distributions from Section~\ref{subsec-experiment-d}; the standard Gaussian distribution $\pi_1$ and the anisotropic Gaussian distribution $\pi_3$ with $\gamma=1/2$. Both BM estimators are constructed over the time interval $[0,Td]$, where $T=10^4$. For the batch size of $\wh{\varsigma^2_h}$, we employ the asymptotically optimal choice in terms of MSE, which is derived explicitly in \cite{Liu+2022}. For $\wh{\varsigma^2_g}$, we set $b/(10^4\sqrt{d})=1.5,2$, respectively. This choice can be suboptimal; optimal batch size selection for our proposed estimator will be an important topic for future research. We report the values of
\begin{equation}\label{eq-varsigma}
    \wh{\varsigma^2}_\text{slow}:=\wh{\varsigma^2_h},\qquad\wh{\varsigma^2}_\text{fast}:={2}/{\wh{\varsigma^2_g}},
\end{equation}
respectively. They are both estimators for the asymptotic variance $\varsigma^2_h$ of the ergodic average $\hat{h}$. This value equals $\varsigma^2_h=8/\sigma^2_F=\sqrt{2\pi}$ for the standard Gaussian target $\pi_1$. For the anisotropic Gaussian target $\pi_3$, we use the estimated asymptotic variances in Section~\ref{subsec-experiment-d} as a proxy for the true values of $\wh{\varsigma^2}_\text{slow}$ and $\wh{\varsigma^2}_\text{fast}$ to compute MSE.

Figure~\ref{fig-Exp3} shows boxplots of $\wh{\varsigma^2}_\text{slow}$ and $\wh{\varsigma^2}_\text{fast}$ over 100 independent runs, together with the MSE of each estimator. For both targets, we observe that $\wh{\varsigma^2}_\text{slow}$ has stable variation across all dimensions, whereas both the variance and the bias of $\wh{\varsigma^2}_\text{fast}$ decrease as the dimension grows. As a result, the MSE ordering reverses, with $\wh{\varsigma^2}_{\text{fast}}$ eventually outperforming $\wh{\varsigma^2}_{\text{slow}}$, and the advantage widens as $d$ becomes large.

\begin{figure}[htbp]\centering
    \includegraphics[width=0.45\textwidth]{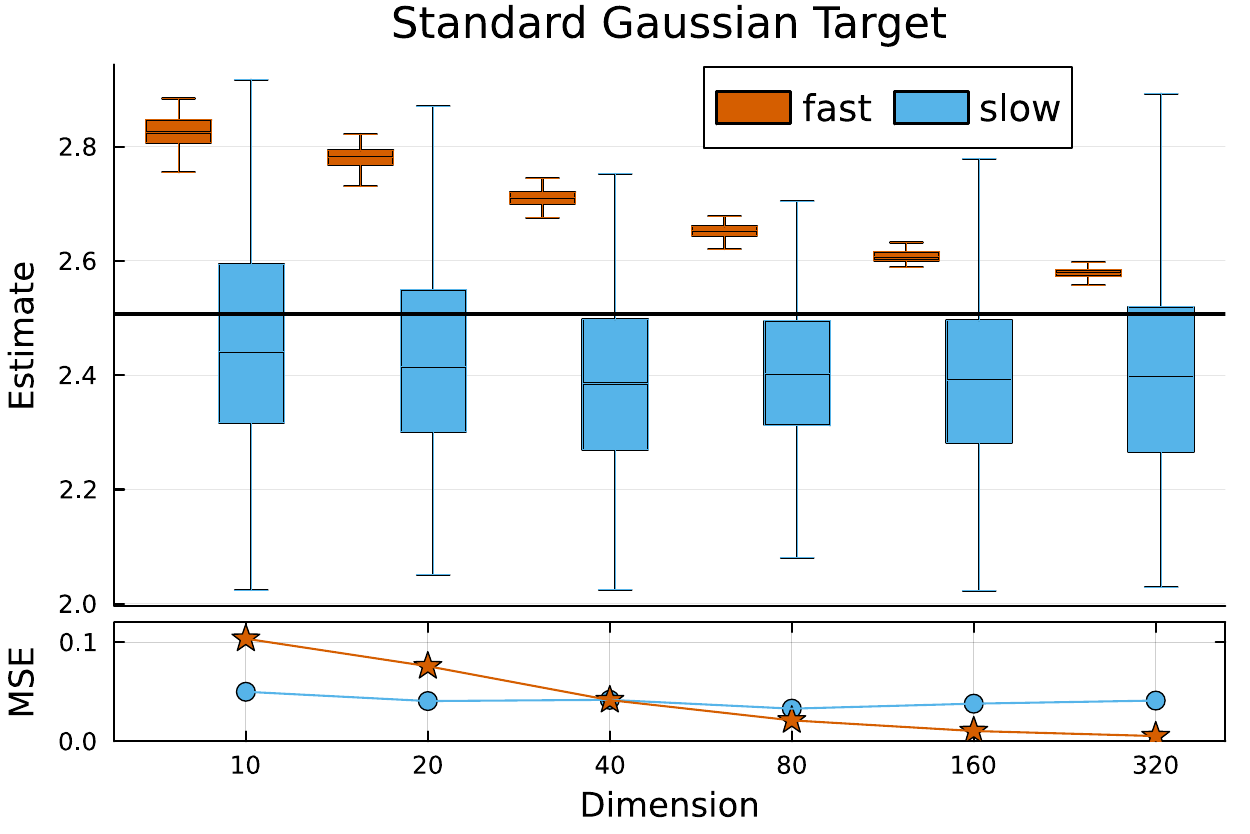}
    \includegraphics[width=0.45\textwidth]{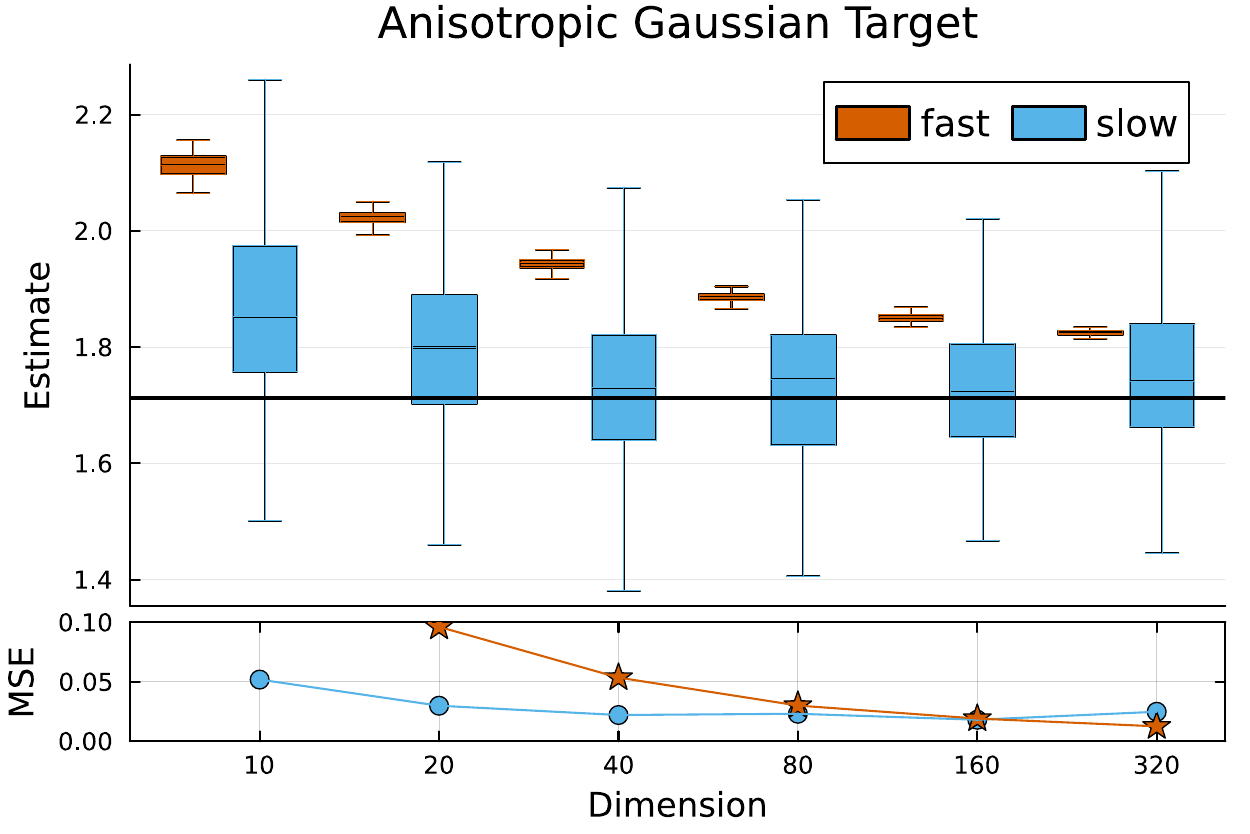}
    \caption{Boxplots of two estimators, $\wh{\varsigma^2}_{\textrm{slow}}$ and $\wh{\varsigma^2}_\textrm{fast}$, over 100 runs against increasing dimensions $d$.
    The targets are either the standard Gaussian distribution $\pi_1(x)\propto\exp(-\abs{x}^2/2)$ (left plot) or the anisotropic Gaussian distribution $\pi_3(x)\propto\exp(-x^\top\Sigma^{-1}x/2)$ with $\Sigma_{ii}=1$ and $\Sigma_{ij}=0.5$ for $i\ne j$ (right plot).
    The black lines in both plots represent the (proxy of) true values.
    The subplot on the bottom represents the mean squared error (MSE) of each estimator.
    }\label{fig-Exp3}
\end{figure}

Our experiment suggests that using $\wh{\varsigma^2_g}$ instead of the usual BM estimator $\wh{\varsigma^2_h}$ results in a lower MSE when the dimension $d$ is large, and that this phenomenon can be observed for a wide variety of targets other than standard Gaussian distributions. As $\wh{g}_T$ is easy to compute on the fly, it may be used to design an adaptive PDMP scheme \citep{Bertazzi-Bierkens2022} by maximizing the estimate internally \citep{Wang+2025}. Adaptive estimation of transport coefficients, such as $\sigma_F$ in our context, has also been considered in molecular dynamics \citep{Jones-Mandadapu2012}.

\section{Conclusion}\label{sec-Conclusion}

To facilitate the development of even more efficient PDMP algorithms, we have proposed analyzing the potential process $Y$, as it constitutes one of the slowest quantities in typical PDMP dynamics. Through its time derivative process $R$, we have established the diffusive scaling limits in Section~\ref{sec-ScalingAnalysis}. Our proof techniques were crucial in establishing the limit for $\rho=0$, the most interesting case for FECMC. The dimensionality scaling $O(d)$ turns out to be identical to BPS. Subsequently, in Section~\ref{sec-Diffusivity}, we provided closed-form formulae for the diffusion coefficients $\sigma_F$ and $\sigma_B$. Our explicit representations were made possible by explicitly solving the Poisson equation $-L^{-1}f=\id$, where $L$ is the generator of $R$. This diffusion coefficient turns out to depend solely on the solution $f$. This observation can facilitate potential improvements in existing PDMP samplers. For instance, maximizing the value, or the expectation \eqref{eq-sigma-to-resolvent}, of this solution $f$ within the state-dependent speed function framework of \cite{Bertazzi-Vasdekis2025} will improve the mixing of the potential process. Additionally, we suggested using the variance of squared increments as a proxy for estimating the asymptotic variance of the ergodic averages of the potential. This exploits the fine structure that resides in the continuous-time trajectory and is unique to PDMP samplers. Its effectiveness in high dimensions was explained theoretically as an application of our scaling analysis results and demonstrated numerically.

The fast proxy estimator $\wh{\varsigma^2}_{\text{fast}}$, defined in Eq.~\eqref{eq-varsigma}, leverages the variability of the velocity variable by using finer batching than standard batch means estimators. The relationship that $2/\wh{\varsigma^2_g}$ matches $\wh{\varsigma^2}_{\text{slow}}$ is a consequence of the OU limit. Accordingly, our estimator relies crucially on the existence of an OU limit for the rescaled potential process. Our experiments suggest this OU limit holds quite universally, thereby supporting the effectiveness of the proposed estimator for the potential function. A more systematic treatment is in order; however, we leave it for future work. This estimator is closely related to variance reduction techniques based on control variates. In this context, the functions $f^F$ and $f^B$ in Eq.~\eqref{eq-sigma-to-resolvent} can be seen as a continuous analogue of the \textit{fishy function} defined and discussed in \cite{Douc+2025,South-Sutton2025}.

As mentioned in discussing our proof strategy in Section~\ref{subsec-proof-strategy}, our limit theorems are closely related to the Kipnis--Varadhan approach to CLTs for additive functionals of Markov processes, initiated in \cite{Kipnis-Varadhan1986}. In that setting, a centered additive functional
\[A_t\coloneq \int^t_0V(X_s)\,\d s\]
of an ergodic Markov process $(X_t)_{t\ge0}$ is considered. In our case, $X$ corresponds to the limiting radial momentum process $R$. With this choice, $A$ coincides with the potential process $Y$ by taking $V(x)=x$. However, our approach differs in two ways from that of \cite{Kipnis-Varadhan1986} and from subsequent extensions to irreversible processes such as \cite{Bhattacharya1982,Cattiaux+2011,Komorowski+2012}. First, the primary interest in these works is the scaling limit of a single process of the form $\alpha_d^{-1}A_{dt}$ for a suitable space-scaling function $d\mapsto\alpha_d$, whereas we study a sequence of processes $(A^d)_{d=1}^\infty$ as we take the limit $d\to\infty$. In other words, we are simultaneously considering a high-dimensional limit and a scaling limit as $d\to\infty$. Second, our development proceeds in the reverse order of the Kipnis--Varadhan program. In typical Kipnis--Varadhan-type arguments, one starts from analytic assumptions ensuring that the resolvent $(\rho-L)^{-1}$ behaves well as $\rho\searrow0$, and then deduces a (functional) central limit theorem together with a Green--Kubo-type formula for the asymptotic variance. In contrast, we first identify the limiting process and its diffusion coefficient in Theorems~\ref{thm-Y} and \ref{thm-Y-positive-rho}, and only afterwards show that the limit $\rho\searrow0$ indeed provides the asymptotic variance (Theorem~\ref{thm-formula-for-sigma}). 

\begin{appendix}

\section{Convergence of the Potential}\label{sec-appendix-A}

The objective here is to prove Theorem~\ref{thm-Y}, deriving the limit of the potential processes $Y^d$ as $d\to\infty$. Recall that $Z^d=(X^d,V^d)$ is the output of FECMC on $\R^d$, a Markov process by design. However, $Y^d$ is not a Markov process with respect to its natural filtration $(\cF_t^d)$, even though it is Markov with respect to the filtration $(\cG_t^d)$ generated by $Z^d$.

\subsection{Convergence of the Piecewise Constant Approximations}\label{sec-appendix-skeleton-chain}

To prove the convergence of $Y^d$, we first prove the convergence of its skeleton processes $\ov{Y}^d$, which are defined below in \eqref{eq-ovY}, composed of the points $Y^d_{S_n/d}$ on the event times $S_n$ rescaled by a factor of $d^{-1}$. After this, we finalize the proof by establishing asymptotic equivalence between $\ov{Y}^d$ and $Y^d$. The jump times of the FECMC process $Z^d$ are denoted by $0=S_0,S_1,S_2,\cdots$, which are $(\cG_t)$-stopping times. Note that they are $(\cF_t^d)$-stopping times as well, as long as $(\cF_t^d)$ is right-continuous, which we assure ourselves here. Using these stopping times $(S_n)$, we define the \textit{skeleton process} $\ov{Y}^d$ by
\begin{equation}\label{eq-ovY}
    \ov{Y}^d_t\coloneq \sum_{n=0}^\infty Y^d_{S_n/d}1_{[\frac{S_n}{d},\frac{S_{n+1}}{d})}(t).
\end{equation}

We now derive its dual predictable projection. We first define $\ov{\cF}^d_{n}\coloneq \cF^d_{S_n/d-}$, which is the $\sigma$-field of events strictly prior to the jump at $S_n/d$. Substituting $n$ with $n(t)\coloneq \max\{n\ge0\mid S_n\le dt\}$, we consider the continuous-time filtration $(\ov{\cF}^d_{n(t)})$. With respect to this filtration, the dual predictable projection of $\ov{Y}^d$ is given by
\begin{equation}\label{eq-Bd}
    B^d_t\coloneq \sum_{n=1}^\infty\E\Square{\Delta\ov{Y}^d_n\,\middle|\,\ov{\cF}^d_{n(t)-1}}1_{[\frac{S_{n-1}}{d},\infty)}(t),
\end{equation}
where $\Delta\ov{Y}^d_n\coloneq Y^d_{S_n/d}-Y^d_{S_{n-1}/d}$. See, e.g., \citet[15.7]{Metivier1982} and \citet[II.3.11]{Jacod-Shiryaev2003}.

This predictable projection $B^d$ constitutes one of the local characteristics, or the triplet, of the semimartingale $\ov{Y}^d$. The other two are the (predictable) quadratic variation $C^d=\brac{M^{d,c}}$ of the continuous part of $M^d\coloneq \ov{Y}^d-B^d-Y_0^d$ and the dual predictable projection $\nu^d$ of the random measure associated with the jumps of $\ov{Y}^d$. These three characteristics $B^d,C^d,\nu^d$ are respectively called the first, second, and third characteristic. For general definitions, see \citet[32.2]{Metivier1982} and \citet[II.2.6]{Jacod-Shiryaev2003}. We first derive the limits of the characteristics $(B^d,C^d,\nu^d)$ and establish Theorem~\ref{thm-Y} based on the framework of \citet[Theorem IX.3.48]{Jacod-Shiryaev2003}.

\subsubsection{Convergence of the First Characteristic}

We start with the first characteristic $B^d$. From now on, we will often add the superscript of $d$ in the stopping time $S_n$ as $S_n^d$ to indicate its ambient dimension, as its distributional properties crucially depend on $d$.
We introduce the notation $T^d_n\coloneq S_n^d-S_{n-1}^d$.

\begin{proposition}[Convergence of the first characteristic $B^d$]\label{prop-A-B}
    Defining $b(y)=-\sqrt{\frac{2}{\pi}}y$, it holds for all $N\ge1$ that
    \[\sup_{1\le n\le Nd}\Abs{\int^{S_n^d/d}_0b(\ov{Y}^d_t)\,\dt-\sum_{i=1}^n\E[\Delta\ov{Y}^d_i|\ov{\cF}^d_{i-1}]}\xrightarrow[d\to\infty]{\P}0,\qquad N\ge1.\]
\end{proposition}
\begin{proof}
    To rewrite the statement, let us define the following two processes:
    \begin{equation}\label{eq-DeltaM}
        \Delta M_n^d\coloneq -\paren{\frac{S_n^d}{d}-\frac{S_{n-1}^d}{d}}\sqrt{\frac{2}{\pi}}\ov{Y}^d_{\frac{S^d_{n-1}}{d}}-\E[\Delta\ov{Y}^d_n|\ov{\cF}^d_{n-1}],
    \end{equation}
    \[M^{d}_n\coloneq \sum_{i=1}^n\Delta M_i^d=-\sum_{i=1}^n\paren{\frac{S_i^d}{d}-\frac{S_{i-1}^d}{d}}\sqrt{\frac{2}{\pi}}\ov{Y}^d_{\frac{S^d_{i-1}}{d}}-\sum_{i=1}^n\E[\Delta\ov{Y}^d_i|\ov{\cF}^d_{i-1}].\]
    The statement is now equivalent to the convergence $\sup_{1\le n\le Nd}\abs{M_n^d}\to0$ in probability for an arbitrary $N\ge1$. As discussed in Section~\ref{subsec-proof-strategy}, one major obstacle in proving this convergence is that $M^d$ is only asymptotically martingale. However, we can employ the following generalized Doob inequality \cite[Corollary 9.7]{Metivier1982}
    \begin{equation}\label{eq-Doob}
        \epsilon^2\P\Square{\sup_{1\le n\le Nd}\abs{M_n^d}>\epsilon}\le\abs{\lambda^d_2}((0,Nd]\times\Omega)+\E\Square{\abs{M_{Nd}^d}^2},
    \end{equation}
    where $\epsilon>0,N\ge1$ and $\lambda_{p}^d$ are the Doléans-Dade measure of the quasimartingale $\abs{M_n^d}^p$ for $p\ge1$; see, e.g., \citet[8.6]{Metivier1982}.

    This inequality will reduce our proof to the convergence of the first and second moments of $M^d_n$. We confirm this by treating the two terms in the right-hand side of \eqref{eq-Doob} one by one.
    The first term can be bounded as follows:
    \begin{align*}
        \abs{\lambda^d_2}((0,Nd]\times\Omega)
        &=\sum_{n=1}^{Nd}\lambda^d_2((n-1,n]\times\Omega)\\
        &=\sum_{n=1}^{Nd}\E\Square{\ABs{\E[(M_n^d)^2|\ov{\cF}^d_{n-1}]-(M_{n-1}^d)^2}}\\
        &=\sum_{n=1}^{Nd}\E\Square{\ABs{\E[(\Delta M_n^d)^2|\ov{\cF}^d_{n-1}]+2M_{n-1}^d\E[\Delta M_n^d|\ov{\cF}^d_{n-1}]}}\\
        &\le\sum_{n=1}^{Nd}\E[(\Delta M_n^d)^2]+2\sum_{n=1}^{Nd}\E\SQuare{\abs{M_{n-1}^d}\ABs{\E[\Delta M_n^d|\ov{\cF}^d_{n-1}]}}.
    \end{align*}
    According to Lemma \ref{lemma-M} that follows, we have $\abs{\lambda_2^d}((0,Nd]\times\Omega)\xrightarrow{d\to\infty}0$ if we can show that $\E[\abs{M_{Nd}^d}^2]\to0$. This fact follows immediately from Lemma \ref{lemma-M} and the following relationship:
    \[\E[M^2_{Nd}]=\sum_{i=1}^{Nd}\E[(\Delta M_n^d)^2]+2\sum_{i>j}\E[\Delta M_i^d\Delta M_j^d].\]
\end{proof}

\begin{lemma}\label{lemma-M}
    The following asymptotic evaluations hold for $\Delta M_n^d$ defined in \eqref{eq-DeltaM}
    \[\E[\Delta M_n^d|\ov{\cF}^d_{n-1}]=O_{\P}(d^{-\frac{3}{2}}),\qquad\E[(\Delta M_n^d)^2|\ov{\cF}^d_{n-1}]=O_{\P}(d^{-2}),\qquad (d\to\infty).\]
\end{lemma}
\begin{proof}
    The proof is straightforward once the asymptotic expressions of $T_i^d$ and $\Delta\ov{Y}^d_i$ are given in Lemma \ref{lemma-T} and \ref{lemma-DeltaYd}, respectively. The calculation proceeds as follows:
    \begin{align*}
    \E[\Delta M_n|\ov{\cF}_{n-1}^d]&=\E\Square{-\sqrt{\frac{2}{\pi}}\paren{\frac{S_n^d}{d}-\frac{S_{n-1}^d}{d}}\ov{Y}^d_{S_{n-1}^d/d}-\E[\Delta\ov{Y}^d_n|\ov{\cF}_{n-1}^d]\,\middle|\,\ov{\cF}_{n-1}^d}\\
    &=-\sqrt{\frac{2}{\pi}}\ov{Y}^d_{S_{n-1}^d/d}\frac{\E[T_n^d|\ov{\cF}_{n-1}^d]}{d}-\E[\Delta\ov{Y}^d_n|\ov{\cF}_{n-1}^d]\\
    &=-\sqrt{\frac{2}{\pi}}\ov{Y}^d_{S_{n-1}^d/d}\Paren{\frac{\sqrt{2\pi}}{d}+O_{\P}(d^{-\frac{3}{2}})}+\frac{2}{d+1}\ov{Y}^d_{S^d_{n-1}/d}+O(d^{-\frac{3}{2}})\\
    &=\ov{Y}^d_{S_{n-1}^d/d}\paren{\frac{2}{d+1}-\frac{2}{d}}+O_{\P}(d^{-\frac{3}{2}}).
  \end{align*}
  The first term in the right-hand side is of order $O_{\P}(d^{-2})$. This implies the first equation in the statement.
  The second equation can be deduced similarly
  \begin{align*}
      \E[(\Delta M_n^d)^2|\ov{\cF}^d_{n-1}]&=\E\Square{\paren{\frac{S_n^d-S_{n-1}^d}{d}}^2\frac{2}{\pi}(\ov{Y}^d_{S_{n-1}^d/d})^2\,\bigg|\,\ov{\cF}^d_{n-1}}\\
      &\qquad+\E\Square{\frac{S_n^d-S_{n-1}^d}{d}\sqrt{\frac{2}{\pi}}\ov{Y}^d_{S_{n-1}^d/d}\E[\Delta\ov{Y}^d_n|\ov{\cF}^d_{n-1}]\,\bigg|\,\ov{\cF}^d_{n-1}}\\
      &\qquad\qquad+\E[\Delta\ov{Y}^d_n|\ov{\cF}^d_{n-1}]^2\\
      &=\frac{\E[(T^d_n)^2|\ov{\cF}^d_{n-1}]}{d^2}\frac{2}{\pi}(\ov{Y}^d_{S_{n-1}^d/d})^2\\
      &\qquad+\frac{\E[T_n^d|\ov{\cF}^d_{n-1}]}{d}\sqrt{\frac{2}{\pi}}\ov{Y}^d_{S^d_{n-1}/d}\E[\Delta\ov{Y}^d_n|\ov{\cF}^d_{n-1}]+\E[\Delta\ov{Y}^d_n|\ov{\cF}^d_{n-1}]^2.
  \end{align*}
  This time, all terms appearing on the rightmost side are of order $O_{\P}(d^{-2})$, given lemmas \ref{lemma-T} and \ref{lemma-DeltaYd}.
\end{proof}

\subsubsection{Convergence of the Modified Second Characteristic}

Here, we consider a modified second characteristic $\wt{C}^d\coloneq \brac{M^d}$. This can be simplified as:
\[\wt{C}_t^d\coloneq \sum_{n=1}^\infty\Paren{\E[(\Delta\ov{Y}^d_n)^2|\ov{\cF}_{n-1}^d]-(\E[\Delta\ov{Y}^d_n|\ov{\cF}_{n-1}^d])^2}1_{[\frac{S_{n-1}}{d},\infty)}(t).\]
Since this quantity converges to the second characteristic of the limiting process, which is the squared diffusion coefficient $\sigma^2_F$ of the OU process \eqref{eq-Y-FECMC} in our case.

\begin{proposition}[Convergence of the Modified Second Characteristic $B^d$]\label{prop-A-C}
    Setting $c\coloneq \frac{8}{\sqrt{2\pi}}$, we have the following convergence:
    \[\int^{S_{Nd}^d/d}_0c\,\dt-\sum_{n=1}^{Nd}\Paren{\E[(\Delta\ov{Y}^d_n)^2|\ov{\cF}_{n-1}^d]-(\E[\Delta\ov{Y}^d_n|\ov{\cF}_{n-1}^d])^2}\xrightarrow[d\to\infty]{\P}0.\]
\end{proposition}
\begin{proof}
    Reformulating the statement using
    \[\Delta N_n^d\coloneq c\frac{T_n^d}{d}-\Paren{\E[(\Delta\ov{Y}^d_n)^2|\ov{\cF}_{n-1}^d]-(\E[\Delta\ov{Y}^d_n|\ov{\cF}_{n-1}^d])^2},\]
    what we need to prove is exactly $\sum_{n=1}^{Nd}\Delta N_n^d\xrightarrow{\P}0$.
    Since $(\E[\Delta\ov{Y}^d_n|\ov{\cF}_{n-1}^d])^2=O_{\P}(d^{-2})$ from Lemma \ref{lemma-DeltaYd}, we obtain
    \[\E[(\Delta N_n^d)^2]=\frac{c^2}{d^2}\E[(T_n^d)^2]-\frac{2c}{d}\E\SQuare{T_n\E[(\Delta\ov{Y}^d_n)^2|\ov{\cF}^d_{n-1}]}+O(d^{-3}).\]
    Therefore, we obtain $\E[(\Delta N_n^d)^2]=O(d^{-2})$ from, again, Lemma \ref{lemma-DeltaYd}.
    Consequently, we are able to finish the proof by the Markov inequality
    \[\P\Square{\Abs{\sum_{n=1}^{Nd}\Delta N_n^d}>\epsilon}\le\sum_{n=1}^{Nd}\frac{\E[(\Delta N_n^d)^2]}{\epsilon^2}\xrightarrow{d\to\infty}0\]
    for every $\epsilon>0$.
\end{proof}

\subsubsection{Finalization}

We are in a position to finalize the proof by establishing the convergence of the skeleton process $\ov{Y}^d$.

\begin{proposition}[Convergence of the skeleton processes]
    The skeleton process $\ov{Y}^d$, defined in \eqref{eq-ovY}, converges in law to $Y$, as stated in Theorem \ref{thm-Y}.
\end{proposition}
\begin{proof}
    The convergence will be established within the framework of \citet[Theorem IX.3.48]{Jacod-Shiryaev2003}. There are six conditions (i)-(vi) to be checked. $b(y)=-(\sqrt{2/\pi})y$ and $c=\sqrt{32/\pi}$, which appeared in Propositions~\ref{prop-A-B} and \ref{prop-A-C}, respectively, are the first and second characteristic of the limiting OU process $Y$, defined in \eqref{eq-Y-FECMC}.
    This fact leaves us with only two conditions, (v) and (vi), to be checked.
    Condition (v) is trivially satisfied from our stationarity assumption.
    Regarding condition (vi), given Propositions~\ref{prop-A-B} and \ref{prop-A-C}, we only need to check the convergence of the third characteristic $\nu^d$. This reduces to proving
    \begin{equation}\label{eq-A-nu}
    \sum_{n=1}^{dN}\P[\abs{\Delta\ov{Y}^d_{S_n}}>\epsilon]\xrightarrow{d\to\infty}0
    \end{equation}
    for every $\epsilon>0$.
    This condition \eqref{eq-A-nu} ensures condition 3.49 of \citet[Theorem IX.3.48]{Jacod-Shiryaev2003} and also condition $[\delta_{\text{loc}}\text{-}D]$ because there exists a choice of $C_1(\R)\subset C_b(\R)$ such that, for every $g\in C_1(\R)$, there exists a number $L>0$ satisfying $g([-L,L])=\{0\}$ and $0\le g(x)\le 1\;(x\in\R)$; see \citet[III.2.7]{Jacod-Shiryaev2003}. For such a function $g$, we have
    \[g*\nu^d_t=\sum_{n=1}^{n(t)}g(\Delta\ov{Y}^d_{S_n})\le\sum_{n=1}^{n(t)}1_{\Brace{\abs{\Delta\ov{Y}^d_{S_n}}>L}}\qquad\therefore\quad\P[g*\nu^d_t]\le\sum_{n=1}^{n(t)}\P[\abs{\Delta\ov{Y}^d_{S_n}}>L],\]
    where $n(t)=\max\{n\ge0\mid S_n\le dt\}$. Taking large enough $N>0$, we can assume $n(t)\le Nd$ with a high probability from Corollary~\ref{cor-number-of-events}. By considering the event $\{n(t)>Nd\}$ separately, it remains to prove \eqref{eq-A-nu}, as in the proof of \citet[Theorem~2.10]{Bierkens+2022}.
    This fact, however, immediately follows from the moment estimate $\E[\abs{\Delta\ov{Y}^d_{S_n}}^3]=O(d^{-3/2})$ implied by Lemma~\ref{lemma-DeltaYd} and Markov's inequality for $p=3$.
\end{proof}

\subsection{Asymptotic Equivalence of the Approximation}\label{sec-appendix-asymptotic-equivalence}

We finalize the proof of Theorem \ref{thm-Y} by proving the asymptotic equivalence of $Y^d$ and $\ov{Y}^d$. The following condition ensures that $Y^d$ and $\ov{Y}^d$ share the same limit; see \citet[Lemma VI.3.31]{Jacod-Shiryaev2003}.

\begin{proposition}[Asymptotic equivalence of $Y^d$ and $\ov{Y}^d$]
    For every positive real number $T>0$,
    \[\sup_{0\le t\le T}\abs{Y^d_t-\ov{Y}^d_t}\xrightarrow[d\to\infty]{\P}0.\]
\end{proposition}
\begin{proof}
    The difference between $Y^d$ and $\ov{Y}^d$ in each time interval $t\in[S_n/d,S_{n+1}/d]$ can be bounded as
    \begin{align*}
        \abs{Y_t^d-Y_{S_n/d}^d}&=\frac{1}{\sqrt{d}}\ABs{\abs{X^d_{td}}^2-\abs{X^d_{S_n}}^2}
        \le\frac{2}{\sqrt{d}}\int^{S_{n+1}}_{S_n}\abs{R_s^d}\,\d s=\frac{2}{\sqrt{d}}\Paren{\abs{R^d_{S_{n}}}T_{n+1}+\frac{T_{n+1}^2}{2}},
    \end{align*}
    where the rightmost side no longer depends on the time $t$. Given that
    \[\ov{Y}^d_t-Y^d_t=\sum_{n=0}^\infty \paren{Y^d_{S_n/d}-Y^d_t}1_{\cointerval{\frac{S_n}{d},\frac{S_{n+1}}{d}}}(t),\]
    we are able to bound, for every $\epsilon>0$ as
    \begin{align*}
        \P\Square{\sup_{0\le t\le T}\abs{Y^d_t-\ov{Y}^d_t}>\epsilon}&
        \le\P[n(T)>Nd]+\P\Square{\bigcup_{n=0}^{Nd}\Brace{\sup_{t\in[S_n/d,S_{n+1}/d]}\abs{Y^d_t-Y^d_{S_n/d}}>\epsilon}}\\
        &\le\P[n(T)>Nd]+\sum_{n=0}^{Nd}\P\Square{d^{-1/2}\paren{2\abs{R_{S_n}^d}T_{n+1}+T_{n+1}^2}>\epsilon},
    \end{align*}
    where $n(t)=\max\{n\ge0\mid S_n\le dt\}$.
    In the right-hand side, the first term converges as $d\to\infty$ from Corollary~\ref{cor-number-of-events} and Markov's inequality for $p=1$. The second term also converges by Markov's inequality for $p=3$ and the fact that $R_{S_n}^d$ and $T_{n+1}$ have bounded absolute moments of any order, conditioned on $\cF_{S_{n}/d-}^d$; see Equations~\eqref{eq-RdSn} and \eqref{eq-T} in Section~\ref{sec-appendix-B-T} and the proof of Lemma~\ref{lemma-W}.
\end{proof}

\subsection{Proof of Theorem~\ref{thm-Y-positive-rho}}

Now that Theorem~\ref{thm-Y} is established, Theorem~\ref{thm-Y-positive-rho} can be established along the same lines as in the proof of Theorem~2.10 in \cite{Bierkens+2022}, combined with results in Section~\ref{sec-appendix-B}; hence, we omit the details.

\section{Asymptotics of Auxiliary Quantities}\label{sec-appendix-B}

This section collects several lemmas concerning the asymptotic properties of certain (conditional) distributions related to the jumps of the FECMC process $Z^d$. The results are heavily used in Section~\ref{sec-appendix-A}.

\subsection{Refreshment Distribution}

The distribution $q^{\tpar,d}$, introduced in \eqref{eq-q-tpar}, arises as the length distribution of a refreshed radial component of the velocity $V^d_{S_n}$ at a jump time $S_n$. We need moment estimates to derive further properties concerning jumps of $Z^d$.

\begin{lemma}[Asymptotic Properties of the New Radial Velocity]\label{lemma-W}
    For the random variable $W^d$ with a probability density $q^{\tpar,d}$, it holds that
    \[\E[W^d]=\frac{1}{2}B\paren{\frac{1}{2},\frac{d+1}{2}}=\sqrt{\frac{\pi}{2}}d^{-\frac{1}{2}}+O(d^{-1})\quad(d\to\infty),\qquad \E[(W^d)^2]=\frac{2}{d+1},\]
    \[\qquad\E[(W^d)^4]=\frac{8}{(d+1)(d+3)},\qquad\E[(W^d)^6]=\frac{48}{(d+1)(d+3)(d+5)},\]
    where $B$ is the beta function $B(z,w)=\int^1_0t^{z-1}(1-t)^{w-1}\,\dt$.
\end{lemma}
\begin{proof}
    Integrating the density $q^{\tpar,d}(v)$, we obtain a simple expression for the corresponding distribution function
    \[F^{\tpar,d}(v)=1-(1-v^2)^{\frac{d-1}{2}}.\]
 Therefore, the second, fourth, and sixth moments can be computed from the fact that, for a uniform random variable $U$ on $[0,1]$,
    \[(F^{\tpar,d})^{-1}(U)=\sqrt{1-(1-U)^{\frac{2}{d-1}}}\deq W,\]
    known as the inverse function method.
    Hence, given $\E[U^k]=(k+1)^{-1}\;(k=0,1,2,\cdots)$, we obtain
    \[\E[(W^d)^2]=1-\E[U^{\frac{2}{d-1}}]=1-\frac{1}{\frac{2}{d-1}+1}=\frac{2}{d+1},\]
    \[\E[(W^d)^4]=1-2\E[(1-U)^{\frac{2}{d-1}}]+\E[(1-U)^{\frac{4}{d-1}}]=\frac{8}{(d+1)(d+3)},\]
    \[\E[(W^d)^6]=1-3\E[(1-U)^{\frac{2}{d-1}}]+3\E[(1-U)^{\frac{4}{d-1}}]-\E[(1-U)^{\frac{6}{d-1}}]=\frac{48}{(d+1)(d+3)(d+5)}.\]
    For the mean of $W^d$, a direct calculation reveals
    \begin{align*}
        \E[W^d]&=\int^1_0xq^{\tpar,d}(x)\,\d x=-\int^1_0x\Paren{(1-x^2)^{\frac
        {d-1}{2}}}'\,\d x\\
        &=\int^1_0(1-x^2)^{\frac{d-1}{2}}\,\d x=\frac{1}{2}B\paren{\frac{1}{2},\frac{d+1}{2}}=\frac{\sqrt{\pi}}{2}\frac{\Gamma\paren{\frac{d+1}{2}}}{\Gamma\paren{\frac{d}{2}+1}}.
    \end{align*}
    The result follows from an asymptotic result regarding the ratio of two Gamma function values
    \[\frac{\Gamma(z+\alpha)}{\Gamma(z+\beta)}=z^{\alpha-\beta}\Paren{1+\frac{(\alpha-\beta)(\alpha+\beta-1)}{2z}+O(\abs{z}^{-2})},\qquad(\abs{z}\to\infty),\]
    which follows from Stirling's series, as presented in \cite{Erdelyi-Tricomi1951}.
\end{proof}

\begin{proposition}[Convergence to the Standard Rayleigh Distribution]\label{prop-convergence-to-Rayleigh}
    The law of $\sqrt{d}W^d$ converges in total variation to $\chi(2)$.
\end{proposition}
\begin{proof}
    The density of $\sqrt{d}W^d$ is $q^{\tpar,d}(y/\sqrt{d})/\sqrt{d}$. This converges to
    \[\frac{q^{\tpar,d}(y/\sqrt{d})}{\sqrt{d}}=(d-1)x\paren{1-\frac{x^2}{d}}^{\frac{d-3}{2}}1_{(0,1)}\paren{\frac{x}{\sqrt{d}}}\xrightarrow{d\to\infty}xe^{-\frac{x^2}{2}}1_{(0,\infty)}(x)\]
    for all $x\in\R$.
\end{proof}

\subsection{Moments of Jump Intervals}\label{sec-appendix-B-T}

We calculate the first and second conditional moments of the jump intervals $T_{n+1}\coloneq S_{n+1}-S_{n}$, given the value of $X_{S_{n}}$ (or equivalently $Y_{S_{n}/d}$). This jump arrival time $T_{n+1}$ is determined as follows. Observe first that the FECMC rate, substituting $\rho=0$ and $U(x)=U^d(x)=\abs{x}^2/2$ in \eqref{eq-lambda}, equals the positive part of $R^d_t=(X^d_t|V^d_t)$, the radial momentum \eqref{eq-Rd}. Note that the time derivative of $R^d_t$ is constant $1$ between jumps, as in our case $\abs{V^d_t}\equiv1$. Therefore, $R^d_t$ is again a piecewise deterministic process. When a jump occurs, a new position is determined as
\begin{equation}\label{eq-RdSn}
    R^d_{S_n}\overset{d}{=}-\abs{X^d_{S_n-}}W^d,\qquad n=1,2,\cdots,
\end{equation}
where $W^d$ is the random variable, independent of $X^d_{S_n-}$, with a density $q^{\tpar,d}$. We calculated the moments of $W^d$ in Lemma \ref{lemma-W}. Conditioned on $S_n$ and $X_{S_n-}^d$, no jump occurs between $t\in [S_n,S_n+\abs{X^d_{S_n-}}W^d]$ as $(R^d_{t})_+=0$ on this time interval. Once $t>S_n+\abs{X_{S_n-}^d}W^d$ holds, $R^d_t$ becomes positive, so a jump may occur with intensity $R_t^d$. Overall, the distribution of $T_{n+1}$ can be represented by
\begin{equation}\label{eq-T}
    T_{n+1}\overset{d}{=}\abs{X^d_{S_{n}-}}W^d+\tau,
\end{equation}
where the random variable $\tau\sim\chi(2)$, independent of $X^d_{S_{n}-}$ and $W^d$, follows a Rayleigh distribution with the scale parameter $1$, satisfying
\begin{equation}\label{eq-tau}
    \E[\tau]=\sqrt{\frac{\pi}{2}},\qquad\E[\tau^2]=2,\qquad\E[\tau^4]=8,\qquad\E[\tau^6]=64.
\end{equation}

\begin{lemma}[Moments of Jump Intervals]\label{lemma-T}
    Conditioned on the filtration $(\cF_t^d)$ generated by $Y^d$, we have
    \[\E[T_n|\cF^d_{S_{n-1}/d-}]=\sqrt{2\pi}+O_{\P}(d^{-\frac{1}{2}}),\qquad(d\to\infty),\]
    \[\E[T^2_n|\cF^d_{S_{n-1}/d-}]=4+\pi+O_{\P}(d^{-\frac{1}{2}}),\qquad(d\to\infty).\]
\end{lemma}
\begin{proof}
    As $Y^d_{S_{n-1}/d-}=y$ corresponds to $\abs{X^d_{S_{n-1}}}^2=y\sqrt{d}+d$ by definition \eqref{eq-Yd}, substituting this into \eqref{eq-T} results in
    \begin{align*}
        \E[T_n|Y^d_{S_{n-1}/d-}=y]&=\E[W^d]\sqrt{y\sqrt{d}+d}+\E[\tau]\\
        &=\sqrt{\frac{\pi}{2}}\sqrt{\frac{y}{\sqrt{d}}+1}+O(d^{-\frac{1}{2}})+\sqrt{\frac{\pi}{2}},
    \end{align*}
    where we used Lemma \ref{lemma-W} and Equation \eqref{eq-tau}.
    The first equation follows immediately, given $Y^d_{S_{n-1}/d-}=O_{\P}(1)$ as $d\to\infty$ for all $n\ge1$. Similarly, the second statement follows from that
    \begin{align*}
        \E[T_n^2|Y^d_{S_{n-1}/d-}=y]&=\paren{\frac{y}{\sqrt{d}}+1}d\E[(W^d)^2]+
        2\E[W^\d\tau]\sqrt{y\sqrt{d}+d}+\E[\tau^2]\\
        &=\paren{\frac{y}{\sqrt{d}}+1}2\cdot\frac{d}{d+1}+2\paren{\frac{\pi}{2}+O(d^{-\frac{1}{2}})}+2.
    \end{align*}
\end{proof}

\subsection{Moments of Skeleton Increments}

In Section \ref{sec-appendix-skeleton-chain}, the dual predictable projection $B^d$ of the skeleton process $\ov{Y}^d$ is proved to be a jump process with increments $\E[\Delta\ov{Y}^d_n|\ov{\cF}^d_{n-1}]$, whose properties we are going to investigate. We first observe that the radial momentum $R^d$ is the time derivative of the potential process $U^d(X^d_t)=\abs{X^d_t}^2/2$. Using this, we have
\begin{align*}
    \Delta\ov{Y}^d_n&=Y^d_{S_n/d}-Y^d_{S_{n-1}/d}=\frac{1}{\sqrt{d}}\Paren{\abs{X^d_{S_n}}^2-\abs{X^d_{S_{n-1}}}^2}\\
    &=\frac{1}{\sqrt{d}}\int^{S_n}_{S_{n-1}}2R^d_{t}\,\dt=\frac{2}{\sqrt{d}}\Paren{R^d_{S_{n-1}}T_n+\frac{T_n^2}{2}}.
\end{align*}
Observe that the quantities $R^d_t,T_n$ appearing in the rightmost expression have already been studied in Equation \eqref{eq-RdSn} and Lemma \ref{lemma-T} respectively.

\begin{lemma}[Moments of Skeleton Increments]\label{lemma-DeltaYd}
    \[\E[\Delta\ov{Y}^d_n|\cF^d_{S_{n-1/d-}}]=-\frac{2}{d+1}Y^d_{\frac{S_{n-1}}{d}-}+O_{\P}(d^{-\frac{3}{2}})\quad(d\to\infty),\]
    \[\E[(\Delta\ov{Y}^d_n)^2|\cF^d_{S_{n-1/d-}}]=\frac{8}{d}+O_{\P}(d^{-\frac{3}{2}}),\qquad\qquad(d\to\infty),\]
    \[\E[(\Delta\ov{Y}^d_n)^3|\cF^d_{S_{n-1/d-}}]=\frac{16}{d^{3/2}}+O_{\P}(d^{-2}),\qquad\qquad(d\to\infty).\]
\end{lemma}
\begin{proof}
    We first express $\Delta\ov{Y}^d_n$ by $R^d_{S_{n-1}}$ and $T_n$. Then substitute \eqref{eq-RdSn} and \eqref{eq-T} within the expectation $\E_y[\cdot]\coloneq \E[\cdot|Y^d_{S_{n-1/d-}}=y]$ to obtain
    \[\Delta\ov{Y}^d_n=\frac{1}{\sqrt{d}}\Paren{-\abs{X_{S_{n-1}}^d}^2(W^d)^2+\tau^2}.\]
    We can proceed as
    \begin{align*}
        \E_y[\Delta\ov{Y}^d_n]&=-\frac{1}{\sqrt{d}}\E[(W^d)^2]\E_y[\abs{X_{S_{n-1}}^d}^2]+\frac{1}{\sqrt{d}}\E[\tau^2]\\
        &=-\frac{1}{\sqrt{d}}\frac{2}{d+1}(y\sqrt{d}+d)+\frac{2}{\sqrt{d}}=-\frac{2}{d+1}y+\frac{2}{\sqrt{d}}\frac{1}{d+1},
    \end{align*}
    where we used the fact that $W^d$ and $\tau$ are independent, the moment estimates of Lemma \ref{lemma-W} and \eqref{eq-tau}, and that
    \[\E[\abs{X^d_{S_{n}}}^2|Y^d_{S_n/d-}=y]=y\sqrt{d}+d.\]
    For the second result, we proceed similarly,
    \begin{align*}
        \E_y[(\Delta\ov{Y}^d_n)^2]&=\frac{\E_y[\abs{X_{S_{n-1}}^d}^4(W^d)^4]}{d}-\frac{2}{d}\E_y[\abs{X_{S_{n-1}}^2}(W^d)^2\tau^2]+\frac{\E[\tau^4]}{d}\\
        &=\frac{8}{d}\frac{d^2}{(d+1)(d+3)}\paren{1+\frac{y}{\sqrt{d}}}^2-\frac{8}{d+1}\paren{1+\frac{y}{\sqrt{d}}}+\frac{8}{d}.
    \end{align*}
    Lastly, we have
    \begin{align*}
        \E_y[(\Delta\ov{Y}^d_n)^3]&=\frac{1}{d^{3/2}}\E_y\Square{\paren{-\abs{X_{S_{n-1}}^d}^2(W^d)^2+\tau^2}^3}\\
        &=d^{-\frac{3}{2}}\Paren{-\E[(W^d)^6](y\sqrt{d}+d)^3+3\E[(W^d)^4]\E[\tau^2](y\sqrt{d}+d)^2\\
        &\qquad-3\E[(W^d)^2]\E[\tau^4](y\sqrt{d}+d)+\E[\tau^6]}\\
        &=d^{-\frac{3}{2}}\Paren{-48+48-48+64+O_{\P}(d^{-1/2})}.
    \end{align*}
\end{proof}

\section{Convergence of the Radial Momentum}\label{sec-appendix-C}

The proof proceeds along the same lines as in Section~\ref{sec-appendix-A}. However, this time $R^d$ possesses a special structure. By focusing on the jump measure $\mu^d$ and its dual predictable projection $\wt{\mu}^d$, first introduced by \cite{Jacod1975}, we are able to obtain much simpler proofs.

\subsection{Convergence of the Third Characteristic}

Let $(M^d_t)_{t\ge0}$ be a time-homogeneous Poisson point process on $(\R_+)^3$ with its intensity measure given by $\ell_+^{\otimes2}\otimes q^{\tpar,d}(v)\,\d v$, where $\ell_+$ is the Lebesgue measure on $\R_+$. The radial momentum process $R^d$ can be expressed as a stochastic integral with respect to $M^d$:
\begin{equation}\label{eq-Rd-integral}
R^d_t=R^d_0+t+\int_{\ocinterval{0,t}\times(\R_+)^2}1_{[0,R^d_{s-}]}(\lambda)(-\abs{X_s^d}v-R^d_{s-})M^d(\d s\d\lambda\d v).
\end{equation}
This integral representation immediately provides us with the canonical semimartingale decomposition; in fact, it is a Doob--Meyer decomposition since $R^d$ is also a quasimartingale. This can be written as $R^d_t={}^\cP\!R^d_t+M^d_t$, where
\[{}^\cP\!R^d_t\coloneq t-\int^t_0\frac{\abs{X_s^d}}{\sqrt{d}}\E[\sqrt{d}W^d](R_{s-}^d)_+\,\d s-\int^t_0(R_{s-}^d)^2\,\d s\]
is the dual predictable projection of $R^d$. Here, $W^d$ is a random variable with density $q^{\tpar,d}$; see Lemma \ref{lemma-W}.

A key observation in analyzing $R^d$ is that its jump measure $\mu^d$, defined by
\[\mu^d(\dt\du)\coloneq \sum_{s\in\R_+}1_{\Brace{\Delta R^d_s\ne0}}\delta_{(s,\Delta X_s)}(\d t\d u)\]
has the dual predictable projection $\wt{\mu}^d$ with respect to the natural filtration of $R^d$, defined by
\begin{equation}\label{eq-wtmu}
    \wt{\mu}^d(\dt\du)\coloneq \int_{\R_+^2}1_{[0,R_{s-}^d]}(\lambda)\delta_{\Brace{-\abs{X_s^d}v-R^d_{s-}}}(\du)\dt\d\lambda q^{\tpar,d}(v)\,\d v.
\end{equation}
We can express the first characteristic ${}^{\cP}\!R^d$ and the modified second characteristic $\brac{M^d}$ solely through the third characteristic $\wt{\mu}^d$ and test functions $f(u)=u,u^2$:
\begin{align}\label{eq-B-by-wtmu}
    u*\wt{\mu}^d_t&=\int_0^tu\wt{\mu}^d(\ds\du)={}^\cP\!R^d_t-t,\\\label{eq-C-by-wtmu}
    u^2*\wt{\mu}^d_t&=\int_0^tu^2\wt{\mu}^d(\ds\du)=\brac{M^d}_t.
\end{align}
Here, $*$ denotes integration with respect to a random measure, as in \citet[Section~II.1]{Jacod-Shiryaev2003}. Therefore, the convergence of the local characteristics can be derived simultaneously from the following lemma. Overall, Proposition~\ref{prop-R} essentially follows from the convergence of the jump distribution, i.e., Proposition~\ref{prop-convergence-to-Rayleigh}.

\begin{lemma}[Strong Convergence of the Third Characteristic]\label{Lemma-wtmu-convergence}
    Let $\wt{\mu}_t$ denote the third characteristic of the limiting process $R_F$ in Proposition~\ref{prop-R}. Then, the convergence
    \[\sup_{0\le t\le T}\abs{f*\wt{\mu}_t^d-f*\wt{\mu}_t\circ R^d}\xrightarrow[d\to\infty]{\P}0\]
    holds for any Lipschitz continuous function $f$ and $f(u)=u^p$ with $p=1,2,\cdots$.
    Here, the underlying probability space is the Skorokhod space $\Omega\coloneq D(\R)$, and $R^d,R_F$ are considered
    to be maps $\Omega\to\Omega$; hence, a composition operator $\circ$ is well-defined.
\end{lemma}
\begin{proof}
    The limiting process $R^F$ defined in Proposition~\ref{prop-R} also has an integral representation:
    \[R_t=R_0+t+\int_{\ocinterval{0,t}\times\R_+^2}1_{[0,R_{s-}]}(\lambda)(-v-R_{s-})M(\ds\d\lambda\d v),\]
    where $M$ is the Poisson random measure with intensity measure $\ell_+^{\otimes2}\otimes\chi(2)$. Observe that the third characteristic satisfies
    \[f*\wt{\mu}_t\circ R^d=\int_{\ocinterval{0,t}\times\R_+}1_{[0,R_{s-}^d]}(\lambda)\E[f(-\tau-R_{s-}^d)]\,\d s\d\lambda,\]
    where $\tau\sim\chi(2)$.
    Combined with Eq. \eqref{eq-wtmu}, we obtain
    \[f*\wt{\mu}_t\circ R^d-f*\wt{\mu}_t^d=\int_{\cointerval{0,t}\times\R_+}1_{[0,R_{s-}^d]}(\lambda)\Paren{\E[f(-\tau-R_{s-}^d)-f(-\abs{X_s^d}W^d-R_{s-}^d)]}\,\d s\d\lambda.\]
    This enables us to bound as
    \[\sup_{0\le t\le T}\abs{f*\wt{\mu}_t\circ R^d-f*\wt{\mu}_t^d}\le\int_{\cointerval{0,T}\times\R_+}1_{[0,R_{s-}^d]}(\lambda)\E[\abs{f(-\tau-R_{s-}^d)-f(-\abs{X_s^d}W^d-R_{s-}^d)}]\,\d s\d\lambda.\]
    When $f(u)=u^p$, by expanding the $p$-th power, the convergence
    \[\E[\abs{f(-\tau-R_{s-}^d)-f(-\abs{X_s^d}W^d-R_{s-}^d)}]\xrightarrow{d\to\infty}0\]
    follows from moment convergence $\E[(\sqrt{d}W^d)^q]\to\E[\tau^q]$ for $1\le q\le p$ and $\abs{X_s^d}/\sqrt{d}\xrightarrow{\text{a.s.}}1$.
    This moment convergence, in turn, follows from weak convergence (Proposition~\ref{prop-convergence-to-Rayleigh}) and uniform integrability for any $p\ge1$, which directly follows from Lemma~\ref{lemma-W} when $p=1,2,\cdots,5$ and can be checked similarly for $p\ge6$.
    The case when $f$ is Lipschitz continuous also follows from $L^1$-convergence $\E[\abs{\tau-\sqrt{d}W^d}]\to0$.
\end{proof}

\subsection{Proof of Proposition~\ref{prop-R}}

\begin{proof}[Proof of Proposition~\ref{prop-R}]
    Again we assume, without losing generality, that the underlying probability space is the Skorokhod space $\Omega=D(\R)$.
    The convergence will be established through \citet[Theorem IX.3.48]{Jacod-Shiryaev2003}.
    The four conditions in (vi) of \citet[Theorem IX.3.48]{Jacod-Shiryaev2003} can be derived from Lemma~\ref{Lemma-wtmu-convergence}.
    The convergence of the first and modified second characteristics, conditions $[\text{Sup-}\beta'_{\text{loc}}]$ and $[\gamma'_{\text{loc}}\text{-}D]$ respectively, follows by taking $f(u)=u$ and $f(u)=u^2$ in the lemma.
    Condition IX~3.49 and $[\delta_{\text{loc}}\text{-}D]$ are also immediate once one realises $f\in C_1(\R)$ can be selected to be always Lipschitz; see \citet[VII.2.7]{Jacod-Shiryaev2003}.
    Condition (v) trivially follows from our stationarity assumption.
    Condition (iv) follows from the expressions \eqref{eq-wtmu}, \eqref{eq-B-by-wtmu}, and \eqref{eq-C-by-wtmu}.
    Condition (iii) is satisfied since the limit $R^d$ is a Markov process with generator $L_F$ \eqref{eq-L_F}, together with \citet[Lemma~IX.4.4]{Jacod-Shiryaev2003}.
    Condition (ii) is also easy as we have
    \[\sup_{\omega\in\Omega}\Paren{\abs{u}^21_{\Brace{\abs{u}>b}}*\wt{\mu}_{t}(\omega)}=\sup_{\omega\in\Omega}\sum_{n=0}^{n(t)(\omega)}\abs{R_{S_n-}(\omega)-R_{S_n}(\omega)}^21_{\Brace{\abs{R_{S_n-}-R_{S_n}}>b}}(\omega).\]
    When this process is stopped at a stopping time $S_a(\omega)\coloneq \inf\{t\ge0\mid\abs{\omega(t)},\abs{\omega(t-)}\ge a\}$, see \citet[IX.3.38]{Jacod-Shiryaev2003}, the right-hand side converges to $0$ as $b\to\infty$, for any $a\ge0$ and $t\ge0$, because the jump distance $\abs{R_{S_n-}(\omega)-R_{S_n}(\omega)}$ cannot exceed $2a$.
    Lastly we prove Condition (i). The total variation process of the predictable projection ${}^\cP\!R$ is
    \[{}^\cP\!R_t=\int^t_0\paren{1-\sqrt{\frac{\pi}{2}}(R_{s-})_+-(R_{s-})_+^2}\,\ds.\]
    Thus, the total variation of a stopped version ${}^\cP\!R^{S_a}$ is strongly majorised, see \citet[Definition~VI.3.34]{Jacod-Shiryaev2003}, by $F_a^1(t)\coloneq t(1+a\sqrt{\pi/2}+a^2)$, for any $a\ge0$.
    Similarly, the stopped modified second characteristic of $R$ is given by
    \[u^2*\wt{\mu}_t=\int^{t}_0\paren{2(R_{s-})_++\sqrt{2\pi}(R_{s-})^2_++(R_{s-})_+^3}\ds,\]
    since $\E[\tau^2]=2$ and $\E[\tau]=\sqrt{\pi/2}$; see \eqref{eq-tau}.
    Therefore, $(u^2*\wt{\mu})^{S_a}$ is majorised by $F_a^2(t)\coloneq t(2a+a^2\sqrt{2\pi}+a^3)$.
    This concludes the proof as $F(a)\coloneq F_a^1+F^2_a$ satisfies Condition (i) of \citet[Theorem IX.3.48]{Jacod-Shiryaev2003}.
\end{proof}

\begin{proof}[Proof of Corollary~\ref{cor-number-of-events}]
    For the jump measure $\mu^d$, it holds that $\E[f*\mu^d_T]=\E[f*\wt{\mu}^d_T]$ for any $T>0$ and bounded measurable function $f\in L^\infty(\R)$; see \citet[Theorem~II.1.8]{Jacod-Shiryaev2003}, \citet[31.3]{Metivier1982} and \citet[Theorem~3.15]{Jacod1975}.
    The result follows immediately:
    \begin{align*}
        \E\Square{\sum_{0\le t\le T}1_{\Brace{\Delta V_t^d\ne0}}}&=\E\Square{\sum_{0\le t\le T}1_{\Brace{\Delta R_t^d\ne0}}}
        =\E[1*\mu^d_T]\\&=\E[1*\wt{\mu}^d_T]=\E\Square{\int^T_0(R_s^d)_+\,\ds}=\int^T_0\E[(R_s^d)_+]\,\d s=\frac{T}{\sqrt{2\pi}},
    \end{align*}
    since $R^d_s\sim N(0,1)$ by stationarity.
\end{proof}

\subsection{Exponential Ergodicity}

\begin{proof}[Proof of Proposition~\ref{prop-exponential-ergodicity}]
    The proof is completed once the drift condition
    \[L^FV\le-V+C\]
    is established for a continuous function $V:\R\to\cointerval{1,\infty}$ and a constant $C\ge0$; see \citet[Theorem~3.2.3]{Kulik2018} and \citet[Theorem~4.1]{Hairer2021}. To be precise, one needs an additional condition requiring that the sublevel set $V^{-1}([1,c])$ is compact for every $c\ge1$ and $P^h$-locally Dobrushin for some $h>0$. However, this follows immediately once we establish the drift condition for
    \begin{equation}\label{eq-V}
        V(x)\coloneq \begin{cases}
            1+x&x\ge0,\\
            e^{-x}&x<0.
        \end{cases}
    \end{equation}
    We will see that the choice of $V$ on $x<0$ is essential, whereas on $x>0$ it could have been chosen otherwise.
    Intuitively, starting negatively away from the origin is highly disadvantageous for mixing of $R^F$, as the process can return to the origin only at unit speed. By contrast, any positive start is almost equally advantageous as a jump to the negative region is allowed anytime. Our choice has been made solely to ensure compactness of $V^{-1}([1,c])$.

    With this choice of $V$, we have $L^FV(x)=-V(x)$ on $x<0$ and
    \[L^FV(x)=1+x\Paren{\E[-\tau]-x}=1-x^2-\sqrt{\frac{\pi}{2}}x\qquad x>0.\]
    Thus, taking $C=2$, we obtain the drift condition with $V$ given in \eqref{eq-V}.
\end{proof}

\section{Formulae for Diffusion Coefficients}\label{sec-appendix-D}

\subsection{Resolvents of Radial Momentum Processes}

\begin{proposition}[Laplace Transforms]
    Let $L^F,L^B$ be the generators of the limiting momentum processes $R^F,R^B$ corresponding to FECMC and BPS, respectively.
    For the identity function $\id(x)=x$,
    $f^F_\rho\coloneq (\rho-L^F)^{-1}\id$ and $f^B_\rho\coloneq (\rho-L^B)^{-1}\id$ can be expressed as
    \[f^F_\rho(x)=\begin{cases}
    e^{\rho x}\paren{k^F(\rho)-\frac{1}{\rho^2}}+\frac{x}{\rho}+\frac{1}{\rho^2},&x\le0,\\
    e^{\rho x+\frac{x^2}{2}}(Nf^F_\rho+1)\int^\infty_xye^{-\paren{\rho y+\frac{y^2}{2}}}\,\d y,& x\ge0,
    \end{cases}\]
    \[Nf^F_\rho\coloneq \int^\infty_0f^F_\rho(-\tau)\tau e^{-\frac{\tau^2}{2}}\,\d\tau=\paren{k^F(\rho)-\frac{1}{\rho^2}}\E[e^{-\rho\tau}]-\frac{1}{\rho}\sqrt{\frac{\pi}{2}}+\frac{1}{\rho^2},\]
    \begin{equation}\label{eq-k-F}
        k^F(\rho)=\frac{\E[e^{-\rho\tau}]}{\rho^2Z^F(\rho)}\paren{-\E[e^{-\rho\tau}]+\rho^2-\rho\sqrt{\frac{\pi}{2}}+1},
    \end{equation}
    where $Z^B(\rho)=1-\E[e^{-\rho\tau}]^2$ and $\tau\sim\chi(2)$ follow the standard Rayleigh distribution, and
    \[f^B_\rho(x)=\begin{cases}
    e^{\rho x}\paren{k^B(\rho)-\frac{1}{\rho^2}}+\frac{x}{\rho}+\frac{1}{\rho^2},&x\le0,\\
    e^{\rho x+\frac{x^2}{2}}\int^\infty_xe^{-\paren{\rho y+\frac{y^2}{2}}}y\paren{1+k^B(\rho)e^{-\rho y}-\frac{y}{\rho}+\frac{1-e^{-\rho y}}{\rho^2}}\,\d y,& x\ge0,
    \end{cases}\]
    \[k^B(\rho)=\frac{1}{\rho^2Z^B(\rho)}\Paren{-\E[e^{-2\rho\tau}]+2(\rho^2+1)\E[e^{-\rho\tau}]-1},\]
    where $Z^B(\rho)=1-\E[e^{-2\rho\tau}]$.
\end{proposition}
\begin{proof}
    For $L^B$, its resolvent operator $(\rho-L^B)^{-1}$ is determined in \citet[Theorem~5.3]{Bierkens-Lunel2022}.
    Therefore, we present calculations leading to the expression for $(\rho-L^F)^{-1}\id$ only. The resolvent equation $(\rho-L^F)f=h$ for $h=\id$ takes the form of
    \[f'(x)=\begin{cases}
        \rho f(x)-x,&x\le0,\\
        (\rho+x)f(x)-x(Nf+1),&x\ge0,
    \end{cases}\]
    which gives the solution
    \[f(x)=\begin{cases}
        e^{\rho x}\paren{k^F(\rho)-\frac{1}{\rho^2}}+\frac{x}{\rho}+\frac{1}{\rho^2},&x\le0,\\
        e^{\rho x+\frac{x^2}{2}}\paren{k^F(\rho)-(Nf+1)\int^x_0ye^{-\paren{\rho y+\frac{y^2}{2}}}\,\d y},& x\ge0,
    \end{cases}\]
    where $k^F(\rho)$ is an integral constant. There is only one choice for $k^F(\rho)$ to maintain the integrability of $f$, namely,
    \[k^F(\rho)=(Nf+1)\int^\infty_0ye^{-\paren{\rho y+\frac{y^2}{2}}}\,\d y.\]
    Substituting $Nf$ into the above formula, we obtain the final result.
\end{proof}

In the above proposition, there is an important quantity that needs to be computed: $\E[e^{-\rho\tau}]$.
This quantity is called the moment generating function of the standard Rayleigh distribution.
This quantity can be readily represented by the (complementary) error function through the following lemma.

\begin{lemma}[Moment Generating Function of the Standard Rayleigh Distribution]
    Let $\tau\sim\chi(2)$ be a standard Rayleigh distribution random variable. The moment generating function simplifies to
    \begin{equation}\label{eq-MGF-to-Omega}
        \E[e^{-\rho\tau}]=1-\sqrt{\frac{\pi}{2}}\rho e^{\frac{\rho^2}{2}}\erfc\paren{\frac{\rho}{\sqrt{2}}}=1-\Omega(\rho),\qquad\rho\ge0,
    \end{equation}
    where $\erfc(z)=\frac{2}{\sqrt{\pi}}\int^\infty_ze^{-t^2}\,\dt$ is the complementary error function, and $\Omega$ is defined in Eq. \eqref{eq-Omega}.
\end{lemma}
\begin{proof}
    A straightforward calculation reveals
    \begin{align*}
        \E[e^{-\rho\tau}]&=\int^\infty_0te^{-\frac{t^2}{2}-\rho t}\,\dt
        =\int^\infty_0(t+\rho)e^{-\frac{t^2}{2}-\rho t}\,\dt-\rho\int^\infty_0e^{-\frac{t^2}{2}-\rho t}\,\dt\\
        &=-\SQuare{e^{-\frac{t^2}{2}-\rho t}}^\infty_0-\rho e^{\frac{\rho^2}{2}}\int^\infty_0 e^{-\frac{(t+\rho)^2}{2}}\,\dt
        =1-\rho e^{\frac{\rho^2}{2}}\sqrt{2}\int^\infty_{\rho/\sqrt{2}}e^{-s^2}\,\ds.
    \end{align*}
\end{proof}

\begin{remark}[On the complementary Error Function and Mills' Ratio]\label{remark-Omega}
    The term $\Omega(\rho)$ on the right-hand side
    \[\Omega(\rho)=\rho e^{\frac{\rho^2}{2}}\int^\infty_\rho e^{-\frac{t^2}{2}}\,\dt=\sqrt{2\pi}\rho e^{\frac{\rho^2}{2}}(1-\Phi(\rho))=\rho M(\rho)\]
    can also be described via Mills' ratio of the Gaussian density $\phi$ and distribution function $\Phi$
    \[M(\rho)=\frac{1-\Phi(\rho)}{\phi(\rho)}=\sqrt{2\pi}e^{\frac{\rho^2}{2}}(1-\Phi(\rho)).\]
    This reinterpretation is exploited in the proof of Corollaries~\ref{cor-continuity-at-0} and~\ref{cor-derivative-of-rho}, as there is a large body of literature on the Mills' ratio and its approximation due to its statistical importance.
    This point of view is also important in numerical computation. Many software packages support the \textit{scaled} complementary error function $\erfcx(x)\coloneq e^{x^2}\erfc(x)$ \citep{Oldham+2009} as a built-in function to prevent numerical overflows \citep{Cody1993,Zaghloul2024}.
    Our Figure~\ref{fig-sigma} is produced through this $\erfcx$, by computing
    \[\Omega(\rho)=\sqrt{\frac{\pi}{2}}\rho e^{\frac{\rho^2}{2}}\erfc\paren{\frac{\rho}{\sqrt{2}}}=\sqrt{\frac{\pi}{2}}\rho\erfcx\paren{\frac{\rho}{\sqrt{2}}}.\]
    This type of quantity also arises in the error performance analysis of digital communication channels \citep{Simon-Alouini1998,Simon-Alouini2004}.
\end{remark}

\subsection{Proof of Theorem~\ref{thm-formula-for-sigma}}

\begin{proof}[Proof of Theorem~\ref{thm-formula-for-sigma}]
    We first prove \eqref{eq-sigma-F}.
    Building upon the decomposition
    \[\frac{\sigma_F^2(\rho)}{8}=\E[R_0^Ff(R_0^F)]=\E\Square{1_{\Brace{R_0^F\le0}}R_0^Ff^F_\rho(R_0^F)}+\E\Square{1_{\Brace{R_0^F\ge0}}R_0^Ff^F_\rho(R_0^F)},\]
    we treat the two terms on the right-hand side separately.
    \begin{description}
        \item[First term.] We have
        \begin{align*}
            \E\Square{1_{\Brace{R_0^F\le0}}R_0^Ff^F_\rho(R_0^F)}
            &=\frac{1}{\sqrt{2\pi}}\int^0_{-\infty}xe^{-\frac{x^2}{2}}\paren{e^{\rho x}\paren{k^F(\rho)-\frac{1}{\rho^2}}+\frac{x}{\rho}+\frac{1}{\rho^2}}\dx\\
            &=\frac{1}{\sqrt{2\pi}}\paren{-\paren{k^F(\rho)-\frac{1}{\rho^2}}\E[e^{-\rho\tau}]+\frac{\E[\tau]}{\rho}-\frac{1}{\rho^2}}\\
            &=\frac{1}{\sqrt{2\pi}}\paren{1-\frac{\rho^2-\rho\sqrt{\frac{\pi}{2}}+\Omega(\rho)}{\rho^2\Omega(\rho)(2-\Omega(\rho))}},
        \end{align*}
        where in the last step we have substituted the expression \eqref{eq-k-F} for $k^F$ and \eqref{eq-MGF-to-Omega} for $\E[e^{-\rho\tau}]$.
        Note that this quantity also equals $-Nf^F_\rho/\sqrt{2\pi}$.
        \item[Second term.] We have
        \begin{align*}
            \E\Square{1_{\Brace{R_0^F\ge0}}R_0^Ff^F_\rho(R_0^F)}&=\frac{1}{\sqrt{2\pi}}\int^\infty_0(Nf^F_\rho+1)xe^{-\frac{x^2}{2}}e^{\rho x+\frac{x^2}{2}}\paren{\int^\infty_xye^{-\paren{\rho y+\frac{y^2}{2}}}\,\d y}\d x\\
            &=\frac{(Nf^F_\rho+1)}{\sqrt{2\pi}}\int^\infty_0xe^{\rho x}\int^\infty_xye^{-\paren{\rho y+\frac{y^2}{2}}}\,\d y\d x\\
            &=\paren{\frac{Nf^F_\rho}{\sqrt{2\pi}}+\frac{1}{\sqrt{2\pi}}}\paren{-\frac{\Omega(\rho)}{\rho^2}+\frac{1}{\rho}\sqrt{\frac{\pi}{2}}},
        \end{align*}
        where in the last step we have applied the integral by parts formula twice.
        Using the expression for $Nf^F_\rho$ obtained in the first term, we arrive at
        \[=\frac{1}{\sqrt{2\pi}}\frac{1}{\rho^2\Omega(\rho)(2-\Omega(\rho))}\paren{\rho^2-\rho\sqrt{\frac{\pi}{2}}+\Omega(\rho)}\paren{-\frac{\Omega(\rho)}{\rho^2}+\frac{1}{\rho}\sqrt{\frac{\pi}{2}}}.\]
        \item[Final step.] Combining the two terms, we conclude that
        \[\frac{\sigma_F^2(\rho)}{8}=\frac{1}{\sqrt{2\pi}}\Square{1-\frac{\paren{\rho^2-\rho\sqrt{\frac{\pi}{2}}+\Omega(\rho)}^2}{\rho^4\Omega(\rho)(2-\Omega(\rho))}}.\]
    \end{description}
    This concludes the proof of Equation~\eqref{eq-sigma-F}.
    The proof for Equation~\eqref{eq-sigma-B} proceeds similarly.
    For the first term, we have
    \begin{align*}
        \E\Square{1_{\Brace{R_0^B\le0}}R_0^Bf^B_\rho(R_0^B)}&=-\frac{1}{\sqrt{2\pi}}\int^\infty_0ye^{-\frac{y^2}{2}}f(-y)\,\d y\\
        &=-\frac{1}{\sqrt{2\pi}}\int^\infty_0ye^{-\frac{y^2}{2}}\paren{e^{-\rho y}\paren{k^B(\rho)-\frac{1}{\rho^2}}-\frac{y}{\rho}+\frac{1}{\rho^2}}\,\d y\\
        &=\frac{1}{\sqrt{2\pi}}\frac{1}{\rho^2}\paren{(1-\rho^2k^B(\rho))\E[e^{-\rho\tau}]+\rho\sqrt{\frac{\pi}{2}}-1}\\
        &=\frac{1}{\sqrt{2\pi}}\frac{1}{\rho^2}\paren{\frac{2\E[e^{-\rho\tau}]-2(\rho^2+1)\E[e^{-\rho\tau}]^2}{1-\E[e^{-2\rho\tau}]}+\rho\sqrt{\frac{\pi}{2}}-1}
    \end{align*}
    For the second term, we have
    \begin{align*}
        \E\Square{1_{\Brace{R_0^B\ge0}}R_0^Bf^B_\rho(R_0^B)}&=\frac{1}{\sqrt{2\pi}}\int^\infty_0xe^{-\frac{x^2}{2}}f(x)\,\d x\\
        &=\frac{1}{\sqrt{2\pi}}\int^\infty_0xe^{-\frac{x^2}{2}}e^{\rho x+\frac{x^2}{2}}\int^\infty_xe^{-\rho y-\frac{y^2}{2}}y\\
        &\qquad\times\paren{1+k^B(\rho)e^{-\rho y}-\frac{y}{\rho}+\frac{1}{\rho^2}-\frac{e^{-\rho y}}{\rho^2}}\,\d y\d x\\
        &=\frac{1}{\sqrt{2\pi}}\frac{1}{\rho^2}\Paren{\E[e^{-\rho\tau}]+k^B(\rho)\E[e^{-2\rho\tau}]\\
        &\qquad-\frac{\E[\tau e^{-\rho\tau}]}{\rho}+\frac{\E[e^{-\rho\tau}]}{\rho^2}-\frac{\E[e^{-2\rho\tau}]}{\rho^2}\\
        &\qquad-1-k^B(\rho)\E[e^{-\rho\tau}]+\frac{\E[\tau]}{\rho}-\frac{1}{\rho^2}+\frac{\E[e^{-\rho\tau}]}{\rho^2}\\
        &\qquad+\rho\E[\tau]+k^B(\rho)\rho\E[\tau e^{-\rho\tau}]\\
        &\qquad-\E[\tau^2]+\frac{\E[\tau]}{\rho}-\frac{\E[\tau e^{-\rho\tau}]}{\rho}
        }
    \end{align*}
    where in the last equality we have carried out integration by parts twice.
    Using the relationship $\rho\E[\tau e^{-\rho\tau}]=1-(1+\rho^2)\E[e^{-\rho\tau}]$, we obtain
    \begin{align*}
        \rho^2\sqrt{2\pi}\E\Square{1_{\Brace{R_0^B\ge0}}R_0^Bf^B_\rho(R_0^B)}&=k^B(\rho)\Paren{\E[e^{-2\rho\tau}]-\E[e^{-\rho\tau}]+\rho\E[\tau e^{-\rho\tau}]}\\
        &\qquad+\rho\E[\tau]+\E[e^{-\rho\tau}]-1-\E[\tau^2]\\
        &\qquad+\frac{1}{\rho}\paren{-2\E[\tau e^{-\rho\tau}]+2\E[\tau]}\\
        &\qquad+\frac{1}{\rho^2}\paren{2\E[e^{-\rho\tau}]-\E[e^{-2\rho\tau}]-1}\\
        &=\sqrt{\frac{\pi}{2}}\left(\rho + \frac{2}{\rho}\right) - 3 - \frac{2}{\rho^2(1-\E[e^{-2\rho\tau}])}\\
        &\qquad\times\paren{(1+\rho^2)\E[e^{-\rho\tau}] - 1} \paren{(2+\rho^2)\E[e^{-\rho\tau}] - 2}.
    \end{align*}
    Combining the two terms, we conclude that
    \begin{align*}
        \rho^2\sqrt{2\pi}\E[R_0^Bf_\rho^B(R_0^B)]&=\frac{1}{\rho^2(1-\E[e^{-2\rho\tau}])}\Paren{-2(1+\rho^2)(2+\rho^2)\E[e^{-\rho\tau}]^2\\
        &\qquad+2\Paren{(2+\rho^2)+2(1+\rho^2)}\E[e^{-\rho\tau}]-4\\
        &\qquad-2\rho^2(1+\rho^2)\E[e^{-\rho\tau}]^2+2\rho^2\E[e^{-\rho\tau}]}+\sqrt{2\pi}\paren{\rho+\frac{1}{\rho}}-4\\
        &=\frac{1}{\rho^2(1-\E[e^{-2\rho\tau}])}\Paren{-4(1+\rho^2)^2\E[e^{-\rho\tau}]^2\\
        &\qquad+8(1+\rho^2)\E[e^{-\rho\tau}]-4}+\sqrt{2\pi}\paren{\rho+\frac{1}{\rho}}-4\\
        &=-\frac{4((1+\rho^2)\E[e^{-\rho\tau}]-1)^2}{\rho^2(1-\E[e^{-2\rho\tau}])}+\sqrt{2\pi}\paren{\rho+\frac{1}{\rho}}-4.
    \end{align*}
    Substituting $\E[e^{-\rho\tau}]=1-\Omega(\rho)$ into the above equation, we arrive at the final result.
\end{proof}

\subsection{Proof of Corollaries}

\begin{proof}[Proof of Corollary~\ref{cor-continuity-at-0}]
    It suffices to show that
    \begin{equation}\label{eq-in-proof-of-continuity-at-0}
        \frac{(\rho^2-\rho\sqrt{\frac{\pi}{2}}+\Omega(\rho))^2}{\rho^4\Omega(\rho)(2-\Omega(\rho))}\to0\qquad(\rho\to0).
    \end{equation}
    One of the clearest ways to show this is to use the Maclaurin series expansion of the Mills' ratio $M(\rho)$; see \citet[41:6:2]{Oldham+2009} and also Remark~\ref{remark-Omega}.
    \[M(\rho)=\sqrt{\frac{\pi}{2}}-\rho+\sqrt{2\pi}\rho^2-\frac{\rho^3}{3}+O(\rho^4),\qquad\rho\to0.\]
    Using this expansion, we see the denominator of the left-hand side of \eqref{eq-in-proof-of-continuity-at-0} is of order $O(\rho^5)$, whereas
    the numerator is of order $O(\rho^6)$.
\end{proof}

\begin{proof}[Proof of Corollary~\ref{cor-derivative-of-rho}]
    The derivative can be computed as
    \begin{align*}
        \frac{\d\sigma^2_F(\rho)}{\d\rho}&=
        \frac{2\sigma_F^2\paren{\rho^2-\rho\sqrt{\frac{\pi}{2}}+\Omega(\rho)}}{\rho^5\Omega^2(\rho)(2-\Omega(\rho))^2}\\
        &\qquad\times\Paren{-\rho^2\Omega^2(\rho)(2-\Omega(\rho))-\paren{\rho^2-\rho\sqrt{\frac{\pi}{2}}}\rho^2(\Omega(\rho)-1)^2\\
        &\qquad+\paren{\rho^2-\rho\sqrt{\frac{\pi}{2}}+\Omega(\rho)}\Omega(\rho)(3-2\Omega(\rho))},
    \end{align*}
    using the laws $M'(\rho)=\Omega(\rho)-1$ and $\Omega'(\rho)=(\rho M(\rho))'=M(\rho)+\rho\Omega(\rho)-\rho$.
    To prove this quantity is negative for any $\rho>0$, we prove that
    \[\paren{\rho^2-\rho\sqrt{\frac{\pi}{2}}+\Omega(\rho)}\Omega(\rho)(3-2\Omega(\rho))<\rho^2\Omega^2(\rho)(2-\Omega(\rho)).\]
    This is equivalent to
    \[\rho^2(\Omega(\rho)-3)(\Omega(\rho)-1)\le\paren{\Omega(\rho)-\rho\sqrt{\frac{\pi}{2}}}(2\Omega(\rho)-3),\]
    and further to
    \[(2-\rho^2)\Omega^2(\rho)+\Omega(\rho)\Paren{4\rho^2-\sqrt{2\pi}\rho-3}+3\sqrt{\frac{\pi}{2}}\rho-3\rho^2>0.\]
    Here we use the lower bound of the Mills' ratio of \cite{Birnbaum1942}, see also \cite{Sampford1953,Yang-Chu2015},
    \[M(\rho)>\frac{2}{\sqrt{\rho^2+4}+\rho}\]
    to establish that
    \begin{align*}
    &(2-\rho^2)\Omega^2(\rho)+\Omega(\rho)\Paren{4\rho^2-\sqrt{2\pi}\rho-3}+3\sqrt{\frac{\pi}{2}}\rho-3\rho^2\\
    &>\paren{\frac{2\rho}{\sqrt{\rho^2+4}+\rho}}^2(2-\rho^2)+\paren{\frac{2\rho}{\sqrt{\rho^2+4}+\rho}}\Paren{4\rho^2-\sqrt{2\pi}\rho-3}+3\sqrt{\frac{\pi}{2}}\rho-3\rho^2\\
    &=\frac{1}{2\rho^2+4+2\rho\sqrt{\rho^2+4}}\Paren{4\rho^2(2-\rho^2)+2\rho(4\rho^2-\sqrt{2\pi}\rho-3)(\sqrt{\rho^2+4}+\rho)\\
    &\qquad+6\rho\paren{\sqrt{\frac{\pi}{2}}-\rho}(\rho^2+2+\rho\sqrt{\rho^2+4})}\\
    &=\frac{\rho}{\rho^2+2+\rho\sqrt{\rho^2+4}}\Paren{\paren{\rho^2+\sqrt{\frac{\pi}{2}}-3}\sqrt{\rho^2+4}-\rho^3+\sqrt{\frac{\pi}{2}}\rho^2-5\rho+6\sqrt{\frac{\pi}{2}}}.
    \end{align*}
    In the rightmost expression, the last factor can be simplified as
    \begin{align*}
    &\hspace{-1cm}\paren{\rho^2+\sqrt{\frac{\pi}{2}}-3}\sqrt{\rho^2+4}-\rho^3+\sqrt{\frac{\pi}{2}}\rho^2-5\rho+6\sqrt{\frac{\pi}{2}}\\
    &>\paren{\rho^2+\sqrt{\frac{\pi}{2}}-3}\rho-\rho^3+\sqrt{\frac{\pi}{2}}\rho^2-5\rho+6\sqrt{\frac{\pi}{2}}\\
    &=2\sqrt{\frac{\pi}{2}}\rho^2-8\rho+6\sqrt{\frac{\pi}{2}}>0.
    \end{align*}
    This is positive since the discriminant is negative:
    \[D=64-4\cdot2\sqrt{\frac{\pi}{2}}\cdot6\sqrt{\frac{\pi}{2}}< 64-48\cdot\frac{3}{2}<0.\]
\end{proof}

\subsection{Proof of Proposition~\ref{prop-asymptotic-MSE}}

Proposition~\ref{prop-asymptotic-MSE} immediately follows from the following stronger result with residual estimates (Lemma~\ref{lemma-MSE}) by applying our weak convergence results (Proposition~\ref{prop-R} and Theorem~\ref{thm-Y}) and the continuous mapping theorem. To facilitate future extensions, we summarize all our assumptions at this point.

\begin{assumption}\label{assumption}
    Recall that the process $(X,V)$ is assumed to be started in stationarity, with the target being standard Gaussian.
    Additionally, assume that
    \[\wh{h}^d_{dT}=\frac{1}{T}\int^{T}_0\frac{Y_t^F}{\sqrt{2}}\,\dt,\qquad\wh{g}^d_T=\frac{1}{T}\int^T_0R_t^F\,\dt,\]
    where $Y^F$ is the OU process given by Eq.~\eqref{eq-Y-FECMC} and $R^F$ given by Eq.~\eqref{eq-L_F}.
    Note that the scaling constant $\sqrt{2}$ appears from the difference between $h$ in Eq.~\eqref{eq-h-g} and $Y^d$ in Eq.~\eqref{eq-Yd}.
\end{assumption}

This corresponds to a first-order approximation in $\wh{h}^d$ and $\wh{g}^d$ as $d\to\infty$. While the standard Gaussian assumption could be relaxed, we do not pursue it here. Our proof remains valid as long as the processes involved, here $Y$ and $R^F$, are sufficiently ergodic.

\begin{lemma}[Asymptotic Estimates of MSEs]\label{lemma-MSE}
    Under Assumption~\ref{assumption}, the following holds for every $d$:
    \[\Var[\wh{h}_{T}^d]=\frac{8}{\sigma^2_F}\frac{d}{T}+O(T^{-2}),\qquad\Var[\wh{g}_{T}^d]=\frac{\sigma^2_F}{4}\frac{1}{T}+O(T^{-2}),\qquad (T\to\infty).\]
\end{lemma}
\begin{proof}[Proof of Lemma~\ref{lemma-MSE}]
    From Lemma~\ref{lemma-variance-integrated-OU} below,
    \[\Var\Square{\int^t_0Y_s\,\d s}=\frac{16}{\sigma^2_F}t-\frac{64}{\sigma^4_F}\paren{1-e^{-\frac{\sigma^2_Ft}{4}}}.\]
    The results for $\wh{h}_T=\frac{1}{T}\int^T_0h(X_s)\,\d s$ immediately follows once one realizes that
    \[2\Var[\wh{h}_T]=\Var\Square{\frac{d}{T}\int^{T/d}_0Y_s\,\d s}=\frac{16}{\sigma^2_F}\frac{d}{T}-\frac{64}{\sigma^4_F}\frac{d^2}{T^2}\paren{1-e^{-\frac{\sigma^2_FT}{4d}}},\]
    where we used Assumption~\ref{assumption}. For $\wh{g}_T$ and $g(x,v)=(x|v)$, such an analytic result is unavailable. However, we can still obtain the result from the exponential decay of the covariance function (Lemma~\ref{lemma-exponential-decay-of-covariance}).
    By a calculation similar to that in Lemma~\ref{lemma-variance-integrated-OU} and the integral by parts formula, for $C(t)\coloneq \Cov[R^F_0,R^F_t]$ we have
    \begin{align*}
        \Var\Square{\frac{1}{\sqrt{T}}\int^T_0R_s^F\,\d s}&=\frac{2}{T}\int^T_0\int^u_0C(u-s)\,\d s\d u=2\int^T_0\paren{1-\frac{u}{T}}C(u)\,\d u.
    \end{align*}
    Therefore, the limiting variance is
    \[\lim_{T\to\infty}\Var\Square{\frac{1}{\sqrt{T}}\int^T_0R_s^F\,\d s}=2\int^\infty_0C(u)\,\d u=\frac{\sigma^2_F}{4}\]
    from the Lebesgue's convergence theorem and Corollary~\ref{cor-continuity-at-0}.
    The estimate
    \[\Abs{\Var[\wh{g}_T]-\frac{\sigma^2_F}{4}\frac{1}{T}}\le\frac{2}{T^2}\int^T_0u\abs{C(u)}\,\d u+\frac{2}{T}\int^\infty_T\abs{C(u)}\,\d u=O(T^{-2})\qquad(T\to\infty)\]
    also follows from the exponential decay of $C$ (Lemma~\ref{lemma-exponential-decay-of-covariance}).
\end{proof}

\begin{lemma}[Variance of Integrated OU Processes]\label{lemma-variance-integrated-OU}
    Let $Y$ be an OU process such that $\d Y_t=-aY_t\,\dt+\sigma\,\d B_t$ for $a,\sigma>0$ and 
    \[Y^*_t\coloneq \int^t_0\,Y_s\,\d s\]
    be an integrated OU process. Then,
    \[\Var[Y^*_t]=\frac{4}{a}t-\frac{4}{a^2}(1-e^{-at}).\]
\end{lemma}
\begin{proof}
    We start by applying Fubini's theorem to obtain
    \[\Var[Y_t^*]=\E\Square{\paren{\int^t_0Y_s\,\d s}^2}=\int_0^t\int^t_0\Cov[Y_s,Y_u]\,\d s\d u.\]
    Substituting the stationary OU covariance $\Cov[Y_s,Y_u]=2e^{-a\abs{u-s}}$ \citep[see, e.g.,][p.37]{Revuz-Yor1999}, we have
    \begin{align*}
        \int^t_0\int^t_0\Cov[Y_s,Y_u]\,\d s\d u&=2\int_0^t\int^t_01_{\Brace{s<u}}2e^{-a(u-s)}\,\d s\d u\\
        &=4\int^t_0\frac{1-e^{-au}}{a}\,\d u=\frac{4}{a}t+\frac{4}{a^2}(e^{-at}-1).
    \end{align*}
    The result also follows from an integral representation given in \citet[2.17]{Barndorff-Nielsen1997}.
\end{proof}

\begin{lemma}[Exponential Decay of $R^F$'s Covariance]\label{lemma-exponential-decay-of-covariance}
    There exist constants $C_1,C_2>0$ such that
    \[\abs{\Cov[R^F_0,R^F_t]}\le C_1e^{-C_2t},\qquad t\ge0.\]
\end{lemma}
\begin{proof}
    From the exponential ergodicity of $R^F$ (Proposition~\ref{prop-exponential-ergodicity}), we have exponential decay of the $\alpha$-mixing coefficient of \cite{Rosenblatt1956} for any skeleton $(R^F_{hn})_{n=1}^\infty$ for $h>0$; see, for example, \citet[Proposition~5.1.1]{Kulik2018}.
    By applying Davydov's inequality \citep{Davydov1968,Rio1993}, we obtain, for every $h>0$,
    \[\abs{\Cov[R^F_0,R^F_{hn}]}\le C\sqrt{\alpha_{hn}}\|R^F_0\|_4^2,\qquad n=1,2,\cdots\]
    where $C>0$ is a constant and $\alpha_{hn}$ is the $\alpha$-mixing coefficient of $R^F_0$ and $R^F_{hn}$.
    Combined together, we obtain the stated result.
\end{proof}

\begin{proof}[Proof of Proposition~\ref{prop-asymptotic-MSE}]
    Given that $\hat{h}^d_{dT}$ can be decomposed into
    \begin{align*}
        \hat{h}^d_{dT}=\frac{1}{T}\int^T_0\frac{Y_t^F}{\sqrt{2}}\,\dt+\frac{1}{T}\int^T_0\frac{Y^d_t-Y^F_t}{\sqrt{2}}\,\dt,
    \end{align*}
    the result (Lemma~\ref{lemma-MSE}) derived under Assumption~\ref{assumption} carries over.
    Same holds for $\hat{g}^d_T$.
\end{proof}

\end{appendix}

\begin{acks}[Acknowledgments]
The authors would like to thank the anonymous referees, an Associate Editor and the Editor for their constructive comments that improved the quality of this paper.
\end{acks}
\begin{funding}
The first author was supported by JST BOOST, Japan Grant Number JPMJBS2412. The second author was supported by JST CREST, Japan Grant Number JPMJCR2115.
\end{funding}

\begin{supplement}
\stitle{Reproducible Julia Code}
\sdescription{Julia scripts to reproduce all figures and numerical experiments.
The same materials are mirrored at \url{https://github.com/162348/paper_HighDimFECMC}.}
\end{supplement}


\bibliographystyle{imsart-nameyear} 
\bibliography{TheBib}       


\end{document}